\documentclass[reqno]{amsart}
\usepackage[utf8]{inputenc}
\usepackage{pxfonts}
\usepackage{amsthm,amsfonts,amstext,amssymb,mathrsfs,amsmath,latexsym} 
\usepackage{cite}
\usepackage{enumerate}
\usepackage{mdwlist}
\usepackage[abs]{overpic}
\usepackage{bm}
\usepackage{hyperref} 
\usepackage{comment} 
\usepackage{tikz}
\usepackage{caption}
\usepackage{subcaption}
\usepackage{color}
\usepackage{esint}

\definecolor{red}{RGB}{255,0,0}
\definecolor{green}{RGB}{0,150,0}
\definecolor{blue}{RGB}{0,0,255}  

\usepackage{cleveref}
 
\crefname{equation}{equation}{equations}
\crefname{figure}{Figure}{Figures}

\newtheorem{thm}{Theorem}[section]

\theoremstyle{definition}
 
\newtheorem{example}[thm]{Example}

\theoremstyle{remark}
\newtheorem{remark}[thm]{Remark} 
\numberwithin{equation}{section} 

\DeclareMathOperator{\re}{Re}

\renewcommand{\Re}{\mathop{\rm Re}}
\renewcommand{\Im}{\mathop{\rm Im}}
\DeclareMathOperator{\supp}{supp} 

\DeclareMathOperator{\const}{const}

\DeclareMathOperator{\cp}{cap}

\newcommand{\mc}[1]{\mathcal{#1}}
\newcommand{\mbb}[1]{\mathbb{#1}}

\newcommand{\N}{\mathbb{N}}
\newcommand{\C}{\mathbb{C}}
\renewcommand{\P}{\mathbb{P}}
\newcommand{\R}{\mathbb{R}}

\title[Do orthogonal polynomials dream of symmetric curves?]{Do orthogonal polynomials dream of symmetric curves?}

\author[A.~Mart\'{\i}nez-Finkelshtein]{A.~Mart\'{\i}nez-Finkelshtein}

\address[AMF]{Department of Mathematics, University of Almer\'{\i}a, Almer\'{\i}a, Spain}

\email{andrei@ual.es}

\author[E.~A.~Rakhmanov]{E.~A.~Rakhmanov}

\address[EAR]{Department of Mathematics, University of South Florida, Tampa (FL), USA}

\email{rakhmano@mail.usf.edu}

 

\begin{document}

\begin{abstract} 
The complex or non-Hermitian orthogonal polynomials with analytic weights are ubiquitous in several areas such as approximation theory, random matrix models, theoretical physics and in numerical analysis, to mention a few. 
Due to the freedom in the choice of the integration contour for such polynomials, the location of their zeros is a priori not clear. Nevertheless, numerical experiments, such as those presented in this paper, show that the zeros not simply cluster somewhere on the plane, but persistently choose to align on certain curves, and in a very regular fashion.

The problem of the limit zero distribution for the non-Hermitian orthogonal polynomials  
is one of the central aspects of their theory. Several important results in this direction have been obtained, especially in the last 30 years, and describing them is one of the goals of the first parts of this paper. However, the general theory is far from being complete, and many natural questions remain unanswered or have only a partial explanation.

Thus, the second motivation of this paper is to discuss some ``mysterious'' configurations of zeros of polynomials, defined by an orthogonality condition with respect to a sum of exponential functions on the plane, that appeared as a results of our numerical experiments. In this apparently simple situation the zeros of these orthogonal polynomials may exhibit different behaviors: for some of them we state the rigorous results, while others are presented as conjectures (apparently, within a reach of modern techniques). Finally, there are cases for which it is not yet clear how to explain our numerical results, and where we cannot go beyond an empirical discussion. 
\end{abstract}

\maketitle


\section{Introduction} \label{sec:historical}

One of the motivations of this paper is to discuss some ``mysterious'' configurations of zeros of polynomials, defined by an orthogonality condition with respect to a sum of exponential functions on the plane, that appeared as a results of our numerical experiments. It turned out that in this apparently simple situation the orthogonal polynomials may exhibit a behavior which existing theoretical models do not explain, or the explanation is not straightforward. In order to make our arguments self-contained, we present a brief outline of the fundamental concepts and known results and discuss their possible generalizations.

The so-called complex or non-Hermitian orthogonal polynomials with analytic weights appear in approximation theory as denominators of rational approximants to analytic functions 
\cite{Gonchar:87,MR1130396} and in the study of continued fractions. 
Recently, non-Hermitian orthogonality found applications in several new areas, for instance in the description of the 
rational solutions to  Painlev\'e equations \cite{Bertola/Bothner14,Balogh15}, in theoretical physics  \cite{MR3090748,MR3234135,MR3388766,Bleher/Deano13,MR3071662} and in numerical analysis \cite{MR2736880}. 

Observe that due to analyticity, there is a freedom in the choice of the integration contour for the non-Hermitian orthogonal polynomials, which means that the location of their zeros is a priori not clear. 
The problem of their limit zero distribution 
is one of the central aspects studied in the theory of orthogonal polynomials, especially in the last few decades. Several important general results in this direction have been obtained, and describing them is one of the goals of the first parts of this paper. However, the general theory is far from being complete, and many natural questions remain unanswered or have only a partial explanation, as some of the examples presented in the second part of this work will illustrate. We will deal with one of the simplest situation that is still posing many open questions.

Complex non-Hermitian orthogonal polynomials are denominators of the diagonal Pad\'e approximants to functions with branch points and thus play a key role in the study of the asymptotic behavior of these approximants, in particular, in their convergence. Since the mid-twentieth century convergence problems for Pad\'e approximants have been attracting wide interest, and consequently, complex orthogonal polynomials have become one of the central topics in analysis and approximation theory. There is a natural historical parallel of this situation with the one occurred in the middle of the ninetieth century, when Pad\'e approximants (studied then as continued fractions) for Markov- and Stieltjes-type functions led to  the introduction of general orthogonal polynomials on the real line. The original fundamental theorems by P.~Chebyshev, A.~Markov and T.~Stieltjes on the subject gave birth to the theory of general orthogonal polynomials.

In 1986 Stahl \cite{MR88d:30004a,Stahl:86} proved a fundamental theorem explaining the geometry of configurations of zeros of non-Hermitian orthogonal polynomials and presented an analytic description of the curves ``drawn'' by the strings of zeros. Those curves are important particular cases of what we now call $S$-curves. They may be defined by the symmetry property of their Green functions or as trajectories of some quadratic differential. 

The fact that the denominators of the  diagonal Pad\'e approximants to an analytic function at infinity  satisfy non-Hermitian orthogonality relations is straightforward and was definitely known in the nineteenth century. Just nobody believed that such an orthogonality could be used to study the properties Pad\'e denominators. Stahl's theorem  showed that complex orthogonality relations may be  effectively used for these purposes, at least for functions with a ``small'' set of singular points, some of them being branch points. This, without any doubt, was a beginning of a new theory of orthogonal polynomials. 

Before the work of Stahl, asymptotics of these polynomials was studied for some subclasses of functions and by appealing to their additional properties. For instance, several important results were obtained by Gonchar \cite{Gonchar73,gonchar:1975} and collaborators in 1970s and the beginning of the 1980s. The geometry of their zero distribution was conjectured (and in partially proved, e.g., for hyperelliptic functions) by J.~Nuttall and collaborators \cite{MR0613842,MR0487173}. Later, in \cite{MR891770}, the case when the logarithmic derivative of the approximated function is   rational (the so-called semiclassical or Laguerre class) was analyzed. The associated orthogonal polynomials, known as the semiclassical or generalized Jacobi polynomials,  satisfy a second-order differential equation, and the classical Liouville--Green (a.k.a~WKB) method may be used to study their strong asymptotics, as it was done by Nuttall, see also a recent paper \cite{MR2964145}. 

Stahl's ideas were considerably extended by Gonchar and Rakhmanov \cite{Gonchar:87} to cover the case when the orthogonality weight depends on the degree of the polynomial, which requires the inclusion of a non-trivial external field (or background potential) in the picture. The curves, describing the location of the strings of zeros of the orthogonal polynomials, feature a symmetry property (the $S$-property, so we call them the $S$-curves), and their geometry is much more involved. The resulting Gonchar--Rakhmanov--Stahl (or GRS) theory, founded by \cite{Gonchar:87,MR88d:30004a,Stahl:86}, allows to formulate statements about the asymptotics of the zeros of complex orthogonal polynomials conditional to the existence of the $S$-curves, which is a non-trivial problem from the geometric function theory. Further contributions in this direction, worth mentioning here, are \cite{MR2737770,MR2854178,KuijlaarsSilva2015,Rakhmanov:2012fk}.

The notion of the $S$-property can be interpreted also in the light of the Deift-Zhou's nonlinear steepest descent method for the Riemann--Hilbert problems \cite{MR2000g:47048}. One of the key steps in the asymptotic analysis is the deformation of the contours, the so-called lens opening, along the  level sets of certain functions. It is precisely the $S$-property of these sets which guarantees that the contribution   on all non-relevant contours becomes asymptotically negligible.

An important further development was a systematic investigation of the critical measures,  presented in \cite{MR2770010}. Critical measures are a wider class that encompasses the equilibrium measures on $S$-curves, see Sect.~\ref{sec:GRS} for the precise definition and further details. One of the contributions of \cite{MR2770010} was the description of their supports in terms of trajectories of certain quadratic differentials (this description for the equilibrium measures with the $S$-property is originally due to Stahl \cite{MR99a:41017}). In this way, the problem of existence of the appropriate $S$-curves is reduced to the question about the global structure of such trajectories. 

Let us finish by describing the content of this paper. 
Section~\ref{sec:examples}  is a showcase of some zero configurations of polynomials of complex orthogonality, appearing in different settings. The presentation is mostly informal, it relies on some numerical experiments, and its goal is mainly to illustrate the situation and eventually to arouse the reader's curiosity.

Sect.~\ref{sec:logpotential} contains a brief overview of some basic definitions from the logarithmic potential theory,
necessary for  the subsequent  discussion, as well as some simple applications of these notions to polynomials. This section is essentially introductory, and a knowledgeable reader may skip it safely. 

In Sect.~\ref{sec:GRS} we present the known basic theorem on asymptotics of complex orthogonal polynomials. We simplify settings as much as possible without losing essential content. The  definitions and results contained here constitute the core of what we call the \emph{GRS theory}.  Altogether, Sects.~\ref{sec:examples}--\ref{sec:GRS} are expository. 

Finally, in Sections~\ref{subsec6.7} and \ref{sec:6} we present some recent or totally new results. For instance, Section~\ref{subsec6.7} is about the so-called vector critical measures, which find applications in the analysis of the Hermite--Pad\'e approximants of the second kind for a couple of power series at infinity of a special form, as well as in the study of the problems tackled in Sect.~\ref{sec:6}. This last section of the paper deals with the orthogonality with respect to \textit{a sum} of two (or more) analytic weights. In order to build some intuition, we present another set of curious numerical results in Sect.~\ref{sec:numerical experiments}, and the title of this paper (partially borrowed from Philip K.~Dick) is motivated by the amazing variety and beauty of possible configurations. As the analysis of these experiments shows even for the simplest model, corresponding to the sum of two exponential weights, in some domain in the parameter space of the problem the standard GRS theory still  explains the observed behavior, while in other domains it needs to be modified or adapted, which leads to some new equilibrium problems. Finally, there are regions in the parameter space where it is not yet clear how to  generalize the GRS theory to explain our numerical results, and we cannot go beyond an empirical discussion.

\section{Zeros showcase} \label{sec:examples}

We start by presenting some motivating examples that should illustrate the choice of the title of this work.

\subsection{Pad\'e approximants}

Pad\'e approximants \cite{MR2963451,MR1383091} are the {\it locally best rational approximants of a power series}; in a broader sense, they are constructive rational approximants  {\it with free poles}.

Let $\mathbb P_n$ denote the set of algebraic polynomials with complex coefficients and degree $\leq n$, and let 
\begin{equation}
\label{eq16}
{f}=\sum_{k=0}^\infty\frac{c_k}{z^{k+1}}
\end{equation}
be a (formal) power series. For any arbitrary non-negative integer $n$ there always exist polynomials~$P_n\in \P_{n-1}$ and~$Q_n\in \mathbb P_n$, $Q_n\not\equiv0$,
satisfying the condition
\begin{equation}
\label{eq17}
R_n(z):=(-P_n+Q_n{f})(z)=\mathcal O\left (\frac{1}{z^{n+1}}\right),\qquad
z \to\infty.
\end{equation}
This equation is again formal and means that $R_n$ (called the \emph{remainder}) is a power series in descending powers of~$z$, starting at least at $z^{-n-1}$. In order to find polynomials~$P_n$ and $Q_n$ we first use condition \eqref{eq17} to determine the coefficients of $Q_n$, after which $P_n$ is just the truncation of $Q_n{f}$ at the terms of nonnegative degree. It is easy to see that condition \eqref{eq17} does not determine the pair $(P_n, Q_n)$ uniquely. Nevertheless, the corresponding rational function $\pi_n(\cdot; f):=P_n/Q_n$ is unique, and it is known as the \emph{(diagonal) Pad\'e approximant to $f$ at $z=\infty$ of degree $n$}.

Hence, the denominator $Q_n$ is the central object in the construction of diagonal Pad\'e approximants, and its poles constitute the main obstruction to convergence of $\pi_n$ in a domain of the complex plane $\C$.

With this definition we can associate a formal orthogonality verified by the denominators $Q_n$ (see, e.g., a recent survey \cite{MR2247778} and the references therein). However, the most interesting theory is developed when $f$ is an analytic germ at infinity. Indeed, if \eqref{eq16} converges for $|z|>R$, then choosing a Jordan closed curve $C$ in $\{z\in \C:\, |z|>R \}$ and using the Cauchy theorem we conclude that
$$
\oint_C z^k Q_n(z) f(z)dz=0, \quad k=0, 1, \dots, n-1.
$$
This condition is an example of a \emph{non-Hermitian orthogonality} satisfied by the denominators $Q_n$.

In particular interesting is the case when $f$ corresponds to an algebraic (multivalued) function, being approximated by intrinsically single-valued rational functions $\pi_n$. As an illustration we plot in Figure~\ref{Fig1Pade} the poles of $\pi_{150}$ for the analytic germs at infinity of two functions,
\begin{equation}
\label{fPadeexample}
f_1(z)=\frac{(z^2-1)^{1/3}}{(z-0.4+0.8i)^{2/3}} \quad \text{and} \quad f_2(z)=\frac{(z^2-1)^{1/5}(z+0.8+0.4i)^{1/5}}{(z-0.4+0.8i)^{3/5}}, 
\end{equation}
both normalized by $f_1(\infty)=f_2(\infty)=1$. These functions belong to the so-called Laguerre class (or are also known as ``semiclassical''): their logarithmic derivatives are rational functions.

\begin{figure}[h]
	\centering 
	\hspace{-3mm} \begin{tabular}{cc}
		\begin{overpic}[scale=0.65]{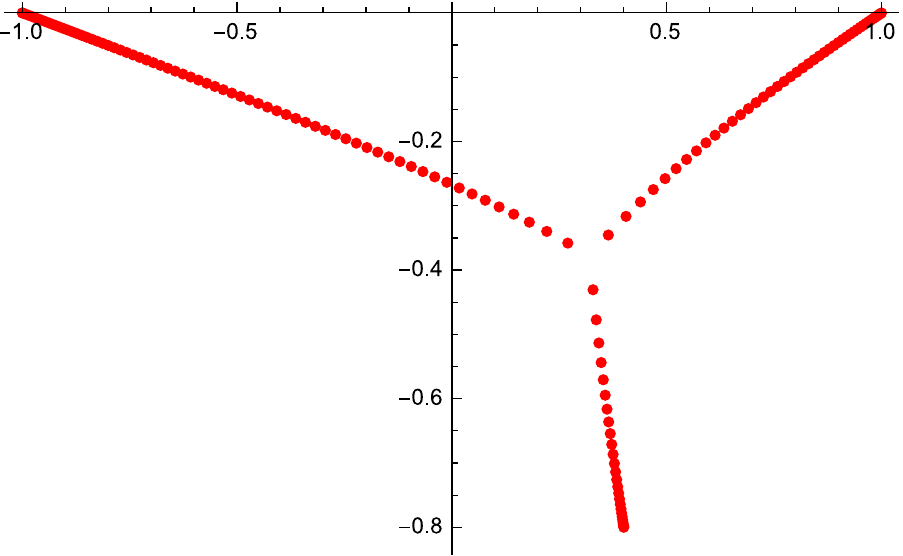}
		\end{overpic} & 
		\begin{overpic}[scale=0.65]{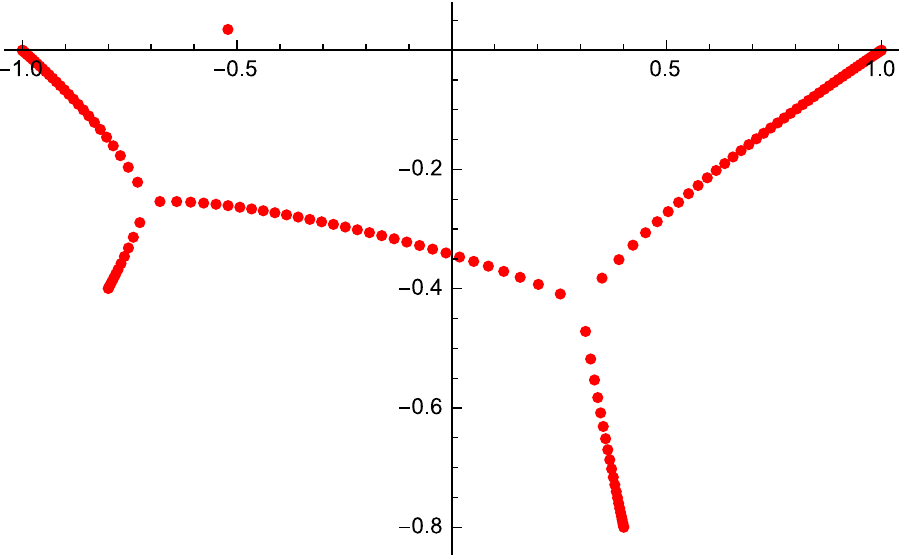}
		\end{overpic}
	\end{tabular}
	\caption{Poles of the Pad\'e approximant $\pi_{150}$ for $f_1$ (left) and $f_2$ (right), see \eqref{fPadeexample}.}\label{Fig1Pade}
\end{figure}
A quick examination of the pictures puts forward two phenomena:
\begin{enumerate}
\item[(i)] generally, poles of $\pi_n$ distribute on $\C$ in a rather regular way. Our eye cannot avoid ``drawing'' curves along which the zeros align almost perfectly.
\item[(ii)] there are some exceptions to this beautiful order: observe a clear outlier on Figure~\ref{Fig1Pade}, right. These ``outliers'' are known as the \emph{spurious poles} of the Pad\'e approximants.
\end{enumerate}

\subsection{Jacobi polynomials}

This is the ``most classical'' family of polynomials, which includes the Chebyshev polynomials as a particular case. They can be defined explicitly (see \cite{MR2723248,szego:1975}) ,
\begin{equation} \label{jacobiDef}
P_{n}^{(\alpha ,\beta )}\left( z\right) =2^{-n}\sum_{k=0}^{n}\left( 
\begin{array}{c}
n+\alpha \\ 
n-k%
\end{array}%
\right) \left( 
\begin{array}{c}
n+\beta \\ 
\ k%
\end{array}%
\right) \left( z-1\right) ^{k}\left( z+1\right) ^{n-k},
\end{equation}
or, equivalently, by the well-known Rodrigues formula
\begin{equation} \label{RodrJac}
P_{n}^{(\alpha ,\beta) }\left( z\right) =\frac{1}{2^{n}n!}\left( z-1\right)
^{-\alpha }\left( z+1\right) ^{-\beta }\left( \frac{d}{dz}\right) ^{n}\left[
\left( z-1\right) ^{n+\alpha }\left( z+1\right) ^{n+\beta }\right] .
\end{equation}

Incidentally, Jacobi polynomials could have been considered in the previous section: for $\alpha, \beta>-1$ they are also denominators of the diagonal Pad\'e approximants (at infinity) to the function
\begin{equation}\label{fForJac}
f(z)=\int_{-1}^1 \frac{(1-t)^\alpha (1+t)^\beta}{t-z}dt.
\end{equation}
In fact, denominators of the diagonal Pad\'e approximants to semiclassical functions as in \eqref{fPadeexample} are known as \emph{generalized Jacobi polynomials}, see \cite{MR891770,MR2964145,Aptekarev:2011fk}.

Clearly, polynomials $P_{n}^{(\alpha ,\beta) }$ are entire functions of the complex parameters $\alpha
,\beta  $. When $\alpha, \beta>-1$ they are orthogonal on $[-1,1]$ with respect to the positive weight $(1-z)^\alpha (1+z)^\beta$, 
\begin{equation}\label{orthogonalityJacobi}
\int_{-1}^1 z^k  P_{n}^{(\alpha ,\beta) }(z) (1-z)^\alpha (1+z)^\beta dz=0, \quad k=0, 1, \dots, n-1,
\end{equation}
and in consequence their zeros are all real, simple, and belong to $(-1,1)$. 

What happens if at least one of the parameters $\alpha, \beta$ \/ is ``non-classical''? In~Figure~\ref{fig:onlyzeros} we depicted the zeros of $P_{n}^{(\alpha ,\beta) }$ for $n=50$, $\alpha = -55 + 5i$ and $\beta=50$. We can appreciate again the same feature: a very regular distribution of the zeros along certain imaginary lines on the plane.

\begin{figure}[htb]
\centering \begin{overpic}[scale=0.7]{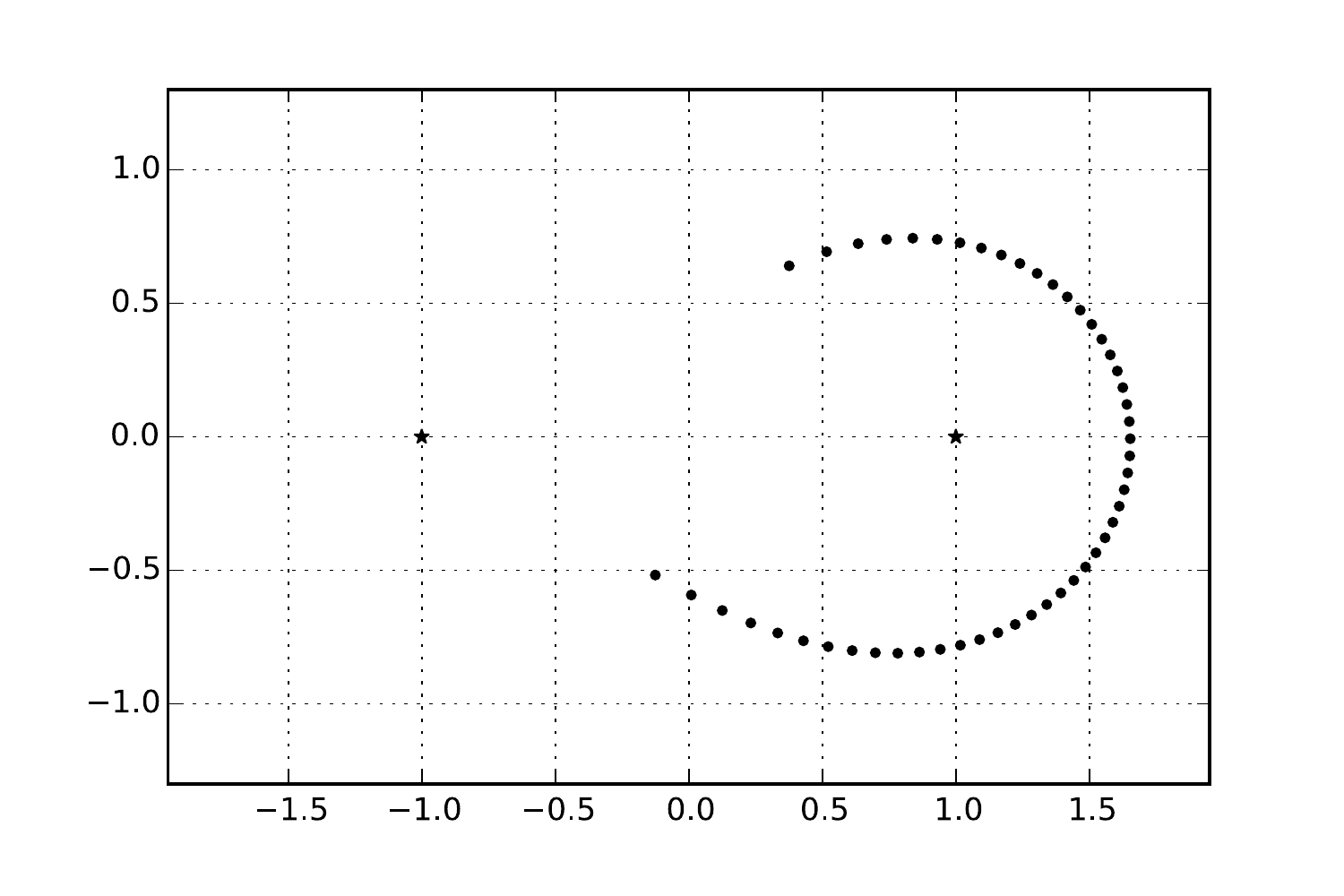}%
\put(29,31){\small $-1 $}
\put(68,33){\small $1 $}
\end{overpic}
\caption{Zeros of the polynomial $P_{50}^{(\alpha ,\beta) }$ for $\alpha = -55 + 5i$ and $\beta=50$.}
\label{fig:onlyzeros}
\end{figure}

\subsection{Heine--Stieltjes polynomials} \label{sec:HSpolyns}

These are a natural generalization of the Jacobi polynomials.
Given a set of pairwise distinct points fixed on the complex plane $\C$,
\begin{equation}\label{A}
\mathcal A=\{a_0, a_1, \dots, a_p\},
\end{equation}
($p\in \N$), and two polynomials,
\begin{equation}\label{defAandB}
    A(z)=\prod_{i=0}^p (z-a_i)\,, \qquad B(z)=\alpha z^p+\text{lower
degree terms} \in \P_p\,, \quad \alpha \in \C ,
\end{equation}
we are interested in the polynomial solutions of the
\emph{generalized Lam\'{e} differential} equation (in algebraic form),
\begin{equation}
\label{DifEq}
A(z)\, y''(z)+ B(z)\, y'(z)- n(n+\alpha -1) V_n(z)\, y(z)=0,
\end{equation}
where $V_n\in \P_{p-1}$; if $\deg V=p-1$, then $V$ is monic.
An alternative perspective on the same problem can be stated in terms of the second order differential operator
$$
\mathcal L[y](z) := A(z)\, y''(z)+ B(z)\, y'(z),
$$
and the associated generalized spectral problem (or multiparameter eigenvalue problem, see \cite{Volkmer1988}),
\begin{equation}\label{ODEclassic}
   \mathcal L[y](z)=n(n+\alpha -1) V_n(z)\, y(z)\,, \quad  n \in \N,
\end{equation}
where $V_n \in \P_{p-1}$ is the ``spectral polynomial''.

If we take $p=1$ in \eqref{DifEq}, we are back in the case of Jacobi polynomials (hypergeometric differential equation). For $p=2$ we get the Heun's equation, which still attracts interest and poses open questions (see \cite{Ronveaux95}). Moreover, denominators of Pad\'e approximants for semiclassical functions are also Heine--Stieltjes polynomials, see e.g.~\cite{MR891770,MR2964145}.

Heine \cite{Heine1878} proved that for every $n \in \N$ there exist
at most
\begin{equation}\label{sigma}
\sigma(n)=\binom{n+p-1}{n}
\end{equation}
different polynomials $V_n$ such that \eqref{DifEq} (or \eqref{ODEclassic}) admits a polynomial solution $y=Q_n \in \P_{n}$. These particular $V_n$ are called \emph{Van Vleck polynomials}, and the corresponding
polynomial solutions $y=Q_n$ are known as \emph{Heine-Stieltjes} (or simply Stieltjes) polynomials; see  \cite{Shapiro2008a} for further details (Fig.~\ref{fig:zerosHS}).

\begin{figure}[htb]
	\centering \begin{overpic}[scale=0.7]{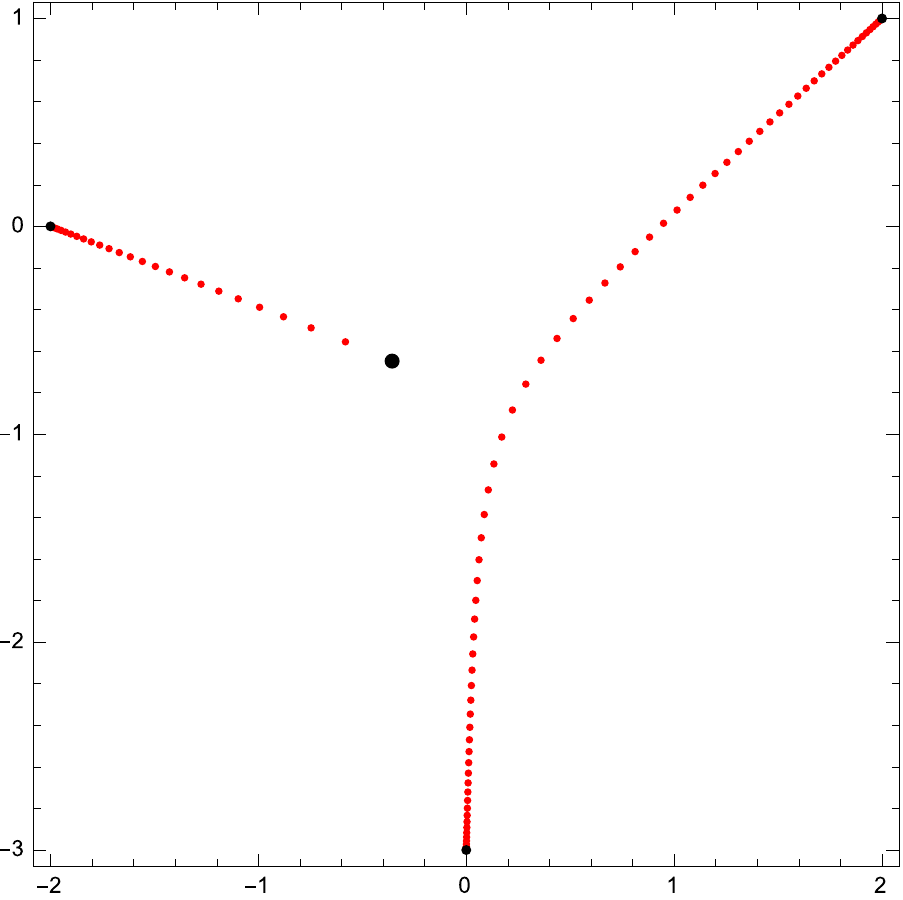}%
	\end{overpic}
	\caption{Zeros of a Heine-Stieltjes polynomial of degree $100$ for $\mathcal A=\{ -3 i, -2, 2 + i\}$. The fat dot is the zero of the corresponding Van Vleck polynomial (of degree $1$).}
	\label{fig:zerosHS}
\end{figure}

Stieltjes discovered an electrostatic interpretation of zeros of the polynomials discussed in \cite{Heine1878}, which attracted general attention to the problem.
He studied the problem \eqref{DifEq} in a particular setting, assuming that $\mathcal A \subset \R$ and that all residues $\rho_k$ in \begin{equation}\label{BoverA}
    \frac{B(x)}{A(x)}=
  \sum_{k=0}^p \frac{\rho_k}{x-a_k}
\end{equation}
are strictly positive (which is equivalent to the assumption that the zeros of $A$
alternate with those of $B$ and that the leading coefficient of
$B$ is positive).
He proved in \cite{Stieltjes1885} (see also \cite[Theorem
6.8]{szego:1975}) that in this case for each $n \in \N$ there are \emph{exactly}
$\sigma(n)$ different Van Vleck polynomials of degree $p-1$ and
the same number of corresponding Heine-Stieltjes polynomials $y$ of
degree $n$, given by all possible ways how the $n$ zeros of $y$
can be distributed in  the $p$ open intervals defined by $\mathcal A$. Obviously, this models applies also in the case $p=1$, i.e.~to the zeros of a Jacobi polynomial. For a more detailed discussion see \cite{MR2345246}, although we describe the electrostatic model of Stieltjes for $p=1$ in the next Section.

\subsection{Hermite--Pad\'e approximants} \label{sec:HPapprox}

Let us return to the Pad\'e approximants. The situation gets more involved if we consider now a \emph{pair} of power series at infinity, say
\begin{equation}
\label{eq16HP}
{f_1}=\sum_{k=0}^\infty\frac{c_k^{(1)}}{z^{k+1}}, \qquad {f_2}=\sum_{k=0}^\infty\frac{c_k^{(2)}}{z^{k+1}},
\end{equation}
fix a non-negative integer $n$, and seek \emph{three} polynomials, $Q_{n,0}$, $Q_{n,1}$ and $Q_{n,2}$, with all $Q_{n,k} \in \P_n$  and not all  $Q_{n,k}\equiv 0$, such that 
\begin{equation} 
R_n(z):=\bigl(Q_{n,0} +Q_{n,1}f_1+Q_{n,2}f_2\bigr)(z) 
=\mathcal O \left( \frac1{z^{2 (n+1)}}\right), \quad z\to\infty;
\label{1.1a}
\end{equation}
$Q_{n,k}$ are the  \emph{type I Hermite--Pad\'e (HP) polynomials}, corresponding to the pair $(f_1, f_2)$ (or more precisely, to the vector $\bm f=(1, f_1, f_2)$), and function $R_n$ defined in \eqref{1.1a} is  again the remainder of the Hermite--Padé approximation to $\bm f$. 

Hermite used in 1858 a construction slightly more general than \eqref{1.1a}, involving $f_k(z) = e^{k/z}$, in order to prove that the number $e$ is transcendental.  HP polynomials play an important role in analysis and have significant  applications in approximation theory, number theory, random matrices, mathematical physics, and other fields. For details and further references see~\cite{MR2963451, MR2963452, MR1478629, MR769985, MR2796829, MR90h:30088, MR2247778}.

\begin{figure}[h]
	\centering 
	\begin{overpic}[scale=1.1]{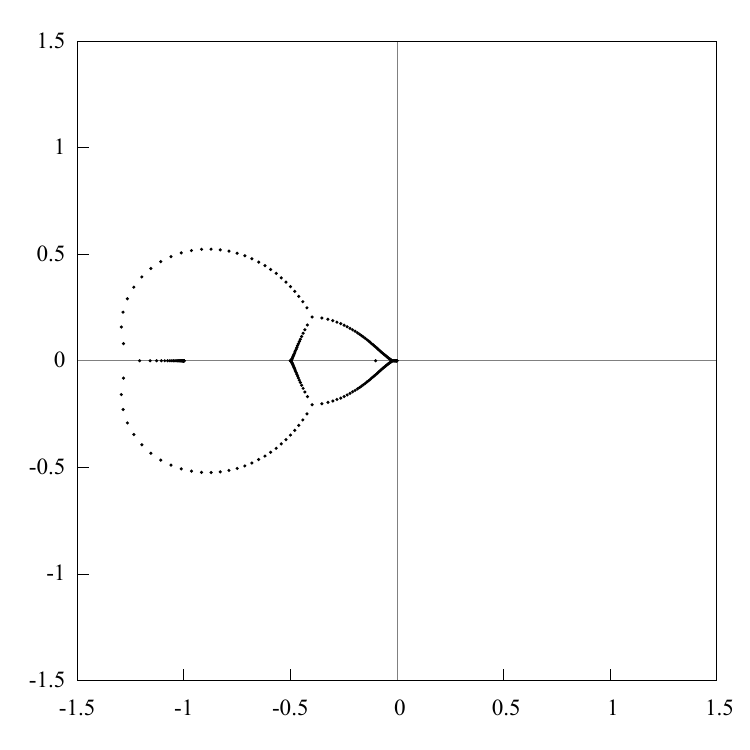}
	\end{overpic} 
	\caption{Zeros of $Q_{250,2}$ for $\bm f=(1, f_1, f_2)$, see \eqref{fHermitePadeexample} (picture courtesy of S.~P.~Suetin).}
	\label{Fig1HermitePade}
\end{figure}

HP polynomials is another classical construction closely related to orthogonal polynomials, but essentially more complicated. An asymptotic theory for such polynomials is not available yet, even though there is a number of separate results. 

As a single illustration of the sophisticated beauty and complexity of this situation, we plot in Figure~\ref{Fig1HermitePade} the zeros of $Q_{250,2}$ for the analytic germs at infinity of two functions,
\begin{equation}
\label{fHermitePadeexample}
f_1(z)=\frac{z^{2/3}}{(z+1)^{1/3}(z-0.5)^{1/3}} \quad \text{and} \quad f_2(z)=\frac{z^{4/3}}{(z+1)^{2/3}(z^2-0.25)^{1/3}}, 
\end{equation}
both normalized by $f_1(\infty)= f_2(\infty)=1$.

As in \eqref{fPadeexample}, functions $f_1$ and $f_2$ are semiclassical, and it is known \cite{MR3088082} that polynomials $Q_{n,j}$ satisfy a differential equation. Observe the interesting and non-trivial configuration of the curves where their zeros lie.

At any rate, let us insist in a general conclusion we can extract from all previous examples: zeros of analyzed polynomials tend to distribute along certain curves on $\C$ in a very regular way that clearly needs an explanation.

\section{Cauchy transforms and the logarithmic potential theory} \label{sec:logpotential}

How can we count the zeros of a polynomial? A trivial observation is that if
$$
p(z)=(z-a_1)\dots (z-a_n)
$$
is a monic polynomial of degree $n$, then its logarithmic derivative $p'/p$ can be written as
\begin{equation}\label{logderiv}
\frac{p'}{p}(z)=\sum_{k=1}^n \frac{1}{z-a_k}=n \int \frac{1}{z-t} d\chi(p)(t),
\end{equation}
where 
$$
\chi(p):=\frac{1}{n}\sum_{k=1}^n \delta_{a_k}
$$
is the \emph{normalized zero-counting measure} for $p$. Observe that we count each zero in accordance with its multiplicity.

The integral in the right-hand side of \eqref{logderiv} is related to the so-called \textit{Cauchy} or \textit{Stieltjes transform} of this measure. In general, given a finite Borel measure $\mu$ on $\C$, its Cauchy transform (in the sense of  the principal value) is
\begin{equation}\label{principalvalue}
C^\mu(z):= \lim_{\epsilon \to 0+} \int_{|z-x|> \epsilon }  \frac{1}{ x-z }\, d\mu(x).
\end{equation} 
Hence, our first identity is
\begin{equation}\label{identity1}
\frac{p'}{np}(z)=-C^{\chi(p)}(z),
\end{equation} 
which is valid as long as $z\notin \{a_1, \dots, a_n\}$, or equivalently, for $z\in \C\setminus\supp(\chi(p))$.

The second, apparently trivial observation is that
$$
-\log|p(z)|=\sum_{k=1}^n \log \frac{1}{|z-a_k|}=\int \log \frac{1}{|z-t|}d\chi(p)(t).
$$
For a finite Borel measure $\mu$ on $\C$ we define its logarithmic potential  
$$
U^\mu(z)  :=  \int \log\frac{1}{|z-t|}\, d\mu(t)\,.
$$
Hence,
\begin{equation}\label{identity2}
|p(z)|^{1/n}=\exp\left(- U^{\chi(p)}(z) \right),
\end{equation} 
valid again for $z\in \C\setminus\supp(\chi(p))$.

The apparently innocent identities \eqref{identity1}--\eqref{identity2} establish the connection between the analytic properties of polynomials and two very important areas of analysis: the harmonic analysis and singular integrals, and the logarithmic potential theory. 

Let us start with a classical example: the zero distribution of the Jacobi polynomials, defined in \eqref{jacobiDef}. 
An already mentioned fact is that the Jacobi polynomial $y=P_n^{(\alpha,\beta)}$ is solution of a second-order differential equation (hypergeometric equation)
\begin{equation}\label{difJac}
(1-z^2) y''(z) +(\beta-\alpha-(\alpha+\beta+2) z) y'(z)+\lambda y(z)=0\,,
\end{equation}
where $\lambda = n (n+\alpha+\beta+1)$, see \cite[Theorem 4.2.2]{szego:1975}. For $\alpha, \beta>-1$ all zeros of $P_n^{(\alpha,\beta)}$ are real, simple, and belong to $(-1,1)$. With $\alpha, \beta>-1$ and $n\to \infty$, consider the normalized zero-counting measures $\chi_n:=\chi(P_n^{(\alpha,\beta)})$. Standard arguments using weak compactness of the sequence $\chi_n$ show that there exist $\Lambda \subset \N$ and a unit measure $\lambda$ on $[-1,1]$ such that 
\begin{equation}\label{weak*}
\chi_n \stackrel{*}{\longrightarrow} 	\lambda \quad \text{for} \quad n \in \Lambda\,.
\end{equation}  
Here we denote by $\stackrel{*}{\longrightarrow}$ the weak-* convergence of measures. 

An expression for the Cauchy transform $C^{\lambda}$ of $\lambda$ can
be obtained in an elementary way directly from (\ref{difJac}). The derivation of the
continued fraction for $h$ from the differential equation appears in the Perron's monograph \cite[\S 80]{Perron:50},
although the original ideas are contained already in the work of
Euler. For more recent applications, see  \cite{Saff80,Faldey:95,MR1805976}.

The differential equation \eqref{difJac} can be rewritten in terms of the function
$$
h_n(z)=-C^{\chi_n}(z)=\frac{p_n'}{np_n}(z) ,
$$
well defined at least for $z\in \C \setminus [-1,1]$. We get
\begin{equation}\label{difTemp}
(1-z^2) \left(
h_n^2- \frac{h_n'}{n} \right)+
\frac{(\alpha+\beta+2)z+\alpha-\beta}{n}\,
h_n+\frac{\alpha+\beta+n+1}{n}=0\,.
\end{equation}
Since zeros of $p_n$ and $p_n'$ interlace on $(-1,1)$, functions $h_n$ are analytic and uniformly bounded in
$\C \setminus [-1,1]$, and by our assumption,
$$
\lim_{n\in \Lambda} h_n(z) = h(z)=-C^{\lambda}(z)
$$
uniformly on compact subsets of (a.k.a.~locally uniformly in) $\C \setminus [-1,1]$. 
Thus, taking limits in
(\ref{difTemp}), we obtain that $h$ is an algebraic function
satisfying a very simple equation:
\[
(1-z^2) h^2(z)+1=0\,.
\]
In other words,
\begin{equation}\label{hExpl}
C^{\lambda}(z) = -\frac{1}{\sqrt{z^2-1}}\,, \quad \C \setminus [-1,1],
\end{equation}
where we take the branch of the square root satisfying $\sqrt{z^2-1}>0$ for $z>1$.
In possession of the additional information that $\supp \lambda \subset
[-1,1]$ we can use the Stieltjes--Perron (or Sokhotsky--Plemelj) inversion formula to recover the measure
$\lambda$: we conclude that $\lambda$ is absolutely continuous on $[-1,1]$, $\supp (\lambda) =
[-1,1]$, and
\begin{equation}\label{equilMeasure}
\lambda'(x)=\frac{1}{\pi \sqrt{1-x^2}}, \quad x\in [-1,1].
\end{equation}
Due to the uniqueness of this expression, we conclude that the limit in \eqref{weak*} holds for the whole sequence $\Lambda=\N$.

Limit \eqref{weak*} with $\lambda$ given by \eqref{equilMeasure} holds actually for families of orthogonal polynomials with respect to a wide class of measures $\mu$ on $[-1,1]$. The explanation lies in the properties of the measure $\lambda$: it is the \emph{equilibrium measure} of the interval $[-1,1]$. In order to define it properly, we need to introduce the concept of the \textit{logarithmic energy} of a Borel measure $\mu $:
\begin{equation}\label{defEnergyContinuous}
E(\mu) :=  \lim_{\varepsilon \to 0}\iint_{|x-y|>\varepsilon} \log \frac{1}{|x-y|}\, d\mu(x) d\mu(y)\,.
\end{equation}
Moreover, given a real-valued function $\varphi\in L^1(|\mu|)$ (the \emph{external field}), we consider also the \textit{weighted energy}
\begin{equation}\label{defWeightedEnergyCont}
E_\varphi(\mu) := E (\mu)+ 2 \int \varphi\, d\mu  \,.
\end{equation}
The electrostatic model of Stieltjes for the zeros of the Jacobi polynomials $P_n^{(\alpha, \beta)}$ says precisely that the normalized zero-counting measure $\chi_n=\chi(P_n^{(\alpha, \beta)})$ minimizes $E_{\varphi_n}(\mu)$, with 
\begin{equation}\label{externalFieldJacobi}
\varphi_n(z)=\frac{\alpha-1}{2n} \log\frac{1}{|z-1|}+\frac{\beta-1}{2n} \log\frac{1}{|z+1|},
\end{equation}
among all discrete measures of the form $\frac{1}{n}\sum_{k=1}^n \delta_{x_k}$ supported on $[-1,1]$. Notice that $\varphi_n(z)$ vanishes asymptotically as $n\to \infty$ for $z\notin \{-1,1\}$, so that it is not surprising that the weak-* limit $\lambda$ of $\chi_n$, given by \eqref{equilMeasure} on $[-1,1]$, minimizes the logarithmic energy among all probability Borel measures living on that interval:
$$
E(\lambda)=\log 2 =\min \left\{ E(\mu):\, \mu \text{ positive, supported on $[-1,1]$, and } \int d\mu=1 \right\}.
$$ 
In the terminology of potential theory, $\mu$ is the \emph{equilibrium measure} of the interval $[-1,1]$, the value $E(\lambda)=\log 2$ is its \emph{Robin constant}, and 
$$
\cp([-1,1])=\exp(-E(\lambda))=\frac{1}{2}
$$
is its \emph{logarithmic capacity}.

As it was mentioned, this asymptotic zero behavior is in a sense universal: it corresponds not only to $P_n^{(\alpha, \beta)}$, but to any sequence of orthogonal polynomials on $[-1,1]$ with respect to a measure $\mu'>0$ a.e. (see,  e.g., \cite{MR892168,MR882423,MR93d:42029}), or even with respect to complex measures with argument of bounded variation and polynomial decay at each point of its support \cite{MR2172271}. We no longer have a differential equation, but we do have the $L^2(\mu)$ minimal property of the orthogonal polynomial, which at the end suffices.

Notice that the arguments above extend easily to the case of parameters $\alpha, \beta$ dependent on $n$. For instance, when the limits
\begin{equation*}
\lim_n \frac{\alpha_n}{n}=\mathfrak a\geq 0,  \quad \lim_n \frac{\beta_n}{n}=\mathfrak b \geq 0
\end{equation*}
exist, the *-limit of the zero-counting measure of $P_n^{(\alpha_n, \beta_n)}$  minimizes the weighted energy with a non-trivial external field
$$
\varphi(z)=\frac{\mathfrak a}{2} \log\frac{1}{|z-1|}+\frac{\mathfrak b}{2} \log\frac{1}{|z+1|}
$$
(see \cite{MR2124460,MR1805976,MR2142296,AMF:2015} for a general treatment of this case). In this situation the sequence $P_n^{(\alpha_n, \beta_n)}$ satisfies \emph{varying orthogonality} conditions, when the weight in \eqref{orthogonalityJacobi} depends on the degree of the polynomial. As it was shown by Gonchar and Rakhmanov \cite{Gonchar:84}, under mild assumptions same conclusions will be valid for general sequences of polynomials satisfying varying orthogonality conditions.

But let us go back to the zeros of an individual Jacobi polynomial $P_n^{(\alpha, \beta)}$. Under the assumption of $\alpha, \beta>-1$, why do they belong to the real line? A standard explanation invokes orthogonality \eqref{orthogonalityJacobi}, but why are we integrating along the real line? The integrand in  \eqref{orthogonalityJacobi} is analytic, so any deformation of the integration path joining $-1$ and $1$ leaves the integral unchanged. Why do the zeros still go to $\R$?

We can modify Stieltjes' electrostatic model to make it ``$\R$-free'' (see \cite{MR2345246}): on a continuum $\Gamma$ joining $-1$ and $1$ find a discrete measures of the form $\frac{1}{n}\sum_{k=1}^n \delta_{x_k}$, supported on $\Gamma$, minimizing $E_{\varphi_n}(\mu)$, with $\varphi_n$ as in \eqref{externalFieldJacobi}; the minimizer is not necessarily unique. Denote this minimal value by $\mc{E}_n(\Gamma)$. Now maximize $\mc{E}_n(\Gamma)$ among all possible continua $\Gamma$ joining $-1$ and $1$. The resulting value is $\log 2$, and the max--min configuration is on $(-1,1)$, given by the zeros of 
$P_n^{(\alpha, \beta)}$.

This max--min ansatz remains valid in the limit $n\to \infty$: among all measures supported on a continuum joining $-1$ and $1$, $\lambda$ in \eqref{equilMeasure} maximizes the minimal energy. Or equivalently, $[-1,1]$ \textit{is the set of minimal logarithmic capacity} among all  continua joining $-1$ and $1$.

If we recall that the Jacobi polynomials are denominators of the diagonal Pad\'e approximants (at infinity) to the function \eqref{fForJac}, using Markov's theorem (see \cite{MR1130396}) we can formulate our conclusion as follows: the diagonal Pad\'e approximants to this $f$ converge (locally uniformly) in $\C\setminus \Gamma$, where $\Gamma$ is the continuum of minimal capacity joining $-1$ and $1$. It is easy to see that $f$ is a multivalued analytic function with branch points at $\pm 1$.

It turns out that this fact is much more general, and is one of the outcomes of the Gonchar--Rakhmanov--Stahl (or GRS) theory.

\section{The GRS theory and critical measures} \label{sec:GRS}

\subsection{$S$-curves}

Let us denote by $H(\Omega)$ the class of functions holomorphic (i.e., analytic and single-valued) in a domain $\Omega$, 
\begin{equation}
\label{eq2.3}
\mc{T}_f:=\{\Gamma\subset\mbb{C}:f\in H(\mbb{C}\setminus \Gamma)\},
\end{equation}
and let $S\in\mc{T}_f$ be defined by the minimal capacity property
\begin{equation}
\label{eq2.4}
\cp(S)=\min_{\Gamma\in\mc{T}_f}\cp(\Gamma).
\end{equation}
Stahl \cite{MR88d:30004a,Stahl:86} proved that the sequence $\left\{\pi_n\right\}$ converges to $f$ in capacity in the complement to $S$,
$$
\pi_n\overset{\cp}{\longrightarrow}f,\quad z\in\mbb{C}\setminus  S,
$$
under the assumption that the set of singularities of $f$ has capacity $0$. The convergence in capacity (instead of uniform convergence) is the strongest possible assertion for an arbitrary function $f$ due to the presence of the so-called spurious poles of the Pad\'e approximants that can be everywhere dense, even for an entire $f$, see \cite{MR87g:41037,MR99a:41017,MR1662718}, as well as Figure \ref{Fig1Pade}.

More precisely, Stahl established the existence of a unique $S\in \mc{T}_f$ of minimal capacity, comprised of a union of analytic arcs, such that the jump $f_+-f_-$ across each arc is $\not\equiv 0$, as well as the fact  that for the denominator $Q_n$ of Pad\'e approximants $\pi_n$ we have $\chi (Q_n)\overset{*}{\ \to\ }\lambda_S$, where $\lambda_S $ is equilibrium measure for $S$. Here and in what follows we denote by $f_+$ (resp., $f_-$) the left (resp., right) boundary values of a function $f$ on an oriented curve.

The original work of Stahl contained not only the proof of existence, but a very useful characterization of the extremal set $S$: on each arc of this set
\begin{equation}\label{Sproperty}
	\frac{\partial U^\nu(z)}{\partial n_+}=\frac{\partial U^\nu(z)}{\partial n_-},
\end{equation}
where $n_\pm $ are the normal vectors to $S$ pointing in the opposite directions. This relation is known as the \textit{$S$-property} of the compact $S$.

Notice that Stahl's assertion is not conditional, and the existence of such a compact set of minimal capacity is guaranteed. In the case of a finite set of singularities, the simplest instance of such a statement is the content of the so-called \textit{Chebotarev's problem} from the geometric function theory about existence and characterization of a continuum of minimal capacity containing a given finite set. It was solved independently by Gr\"otzsch and Lavrentiev in the 1930s.
A particular case of Stahl's results, related to  Chebotarev's problem, states that 
given a finite set $\mathcal{A}=\left\{a_1,\dotsc,a_p\right\}$ of $p\ge2$ distinct points in $\mbb{C}$ there exist a unique set $S$
	$$
	\cp(S)\ =\ \min_{\Gamma\in\mc{T}}\cp(\Gamma),
	$$
	where $\mc{T}$ is the class of continua $\Gamma\subset\mbb{C}$ with $\mathcal {A}\subset\Gamma$. The complex Green function $G(z)=G(z,\infty)$ for $\ \overline{\mbb{C}}\setminus  S$ has the form
	$$
	G(z)=\int^z_a\sqrt{{V(t)}/{A(t)}}\,dt,
	\quad V(z)=z^{p-2}+\dotsb\in\mbb{P}_{p-2}.
	$$
	where $ A(z) = (z-a_1) \dots(z-a_p\ )$ and $V$ is a polynomial uniquely defined by $\mathcal A$. In particular,  we have
	$$
	S=\{z:\Re G(z)=0\},
	$$
and \eqref{Sproperty} is an immediate consequence of these expressions. Another consequence is that $S$ is a union of arcs of critical trajectories of the quadratic differential $-(V/A)(dz)^2$. This is also the zero level of the (real) Green function $g(z)=\Re G(z)$ of the two-sheeted Riemann surface for $\sqrt{V/A}$.

In order to study the limit zero distribution of Pad\'e denominators $Q_n$ Stahl \cite{Stahl:86} created an original potential theoretic method based directly on the non-Hermitian orthogonality relations
$$
\oint_\Gamma Q_n(z)\,z^k\,f(z)\,dz=0,\quad
k=0,1,\dotsc,n-1, \quad\Gamma\in\mc{T}_f,
$$
satisfied by these polynomials; incidentally, he also showed for the first time how to deal with a non-Hermitian orthogonality. The method was further developed by Gonchar and Rakhmanov in \cite{Gonchar:87} for the case of varying orthogonality (see also \cite{MR2854178}). The underlying potential theoretic model in this case must be modified by including a non--trivial external field $\varphi$.
If the set $S$ on the plane is comprised of a finite number of piece-wise analytic arcs, we say that it exhibits the \textit{$S$-property in the external field $\varphi$} if 
\begin{equation}\label{SpropertyExtF}
\frac{\partial (U^\lambda+\varphi)(z)}{\partial n_+}=\frac{\partial (U^\lambda+\varphi)(z)}{\partial n_-},\quad z\in \supp (\lambda)\setminus e,
\end{equation}
where $\lambda=\lambda_{S,\varphi}$ is now the minimizer of the weighted energy \eqref{defWeightedEnergyCont} among all probability measures supported on $S$, and $\cp(e)=0$. In other words, $\lambda$ is the equilibrium measure of $S$ in the external field $\varphi$, and can be characterized by the following variational (equilibrium) conditions: there exists a constant $\omega$ (the \textit{equilibrium constant}) such that
\begin{equation} \label{Equ-2}
\left(U^\lambda+\varphi\right)(z) \begin{cases}
=\omega, & z\in\supp(\lambda), \\
\ge \omega, & z\in S.
\end{cases}
\end{equation}
Equation \eqref{Equ-2} uniquely defines both the probability measure $\lambda $ on $S$ and the constant $\omega=\omega_{\varphi,S}$.

The pair of conditions \eqref{SpropertyExtF}--\eqref{Equ-2} has a standard electrostatic interpretation, which turns useful for understanding the structure of the $S$-configurations. Indeed, it follows from \eqref{Equ-2} that distribution of a positive charge presented by $\lambda$ is in  equilibrium on the fixed conductor S, while the $S$-property of compact $S$  in \eqref{SpropertyExtF} means that forces acting on the element of charge at $z\in \supp(\lambda)$ from both sides of $S$ are equal. So, the  distribution  $\lambda$ of an $S$-curve will remain in equilibrium if we remove the condition (``scaffolding'') that the charge belongs to $S$ and make the whole plane a conductor (except for a few exceptional insulating points, such as the endpoints of some of arcs in the support of $\lambda$). In other words, $\lambda$ is a distribution of charges which is in an (unstable) equilibrium in a conducting domain. 

Let $\Omega$ be a domain in $\C$, $\Gamma$ a compact subset of $\Omega$ of positive capacity, $f\in H(\Omega\setminus \Gamma)$, and let the sequence $\Phi_n(z)\in H(\Omega)$. Assume that polynomials $Q_n(z)=z^n+\dotsb$ are defined by the non-Hermitian orthogonality relations 
\begin{equation}
\label{eq1.7}
\oint_S Q_n(z)z^k w_n(z)dz=0,\quad k=0,1,\dots,n-1,
\end{equation}
where 
\begin{equation}\label{singleExponent}
w_n(z):=e^{-2n\Phi_n(z)} f(z),
\end{equation}
the integration goes along the boundary of $\overline \C \setminus \Gamma$ (if such integral exists, otherwise integration goes over an equivalent cycle in $\overline \C \setminus \Gamma$). 

A slightly simplified version of one of the main results of Gonchar and Rakhmanov \cite{Gonchar:87} is the following: 
\begin{thm}
	\label{GRSthm}
Assume that $\Phi_n$ converge locally uniformly in $\Omega$ (as $n\to \infty$) to a function $\Phi(z)$. If $\Gamma$ has the $S$-property in $\varphi=\Re\Phi(z)$ and if the complement to the support of the equilibrium measure $\lambda=\lambda_{S,\varphi}$ is connected, then $\chi(Q_n)\overset{*}{ \longrightarrow}\lambda
	$.
\end{thm}
Theorem~\ref{GRSthm} was proved in \cite{Gonchar:87}, where it was called ``generalized Stahl's theorem''. 
Observe that unlike the original Stahl's theorem, its statement is conditional: \emph{if} we are able to find a compact set with the $S$-property (in a harmonic external field) and connected complement, \emph{then} the weak-* convergence is assured.  Under general assumptions on the class of integration paths $\mc F$ and in the presence of a non--trivial external field, neither existence nor uniqueness of a set with the $S$-property are guaranteed. 

A general method of solving the existence problem is based on the maximization of the equilibrium energy, which is inspired by the minimum capacity characterization (in the absence of an external field), see \eqref{eq2.4}. More exactly, consider the problem of finding $S \in \mc F$ with the property
\begin{equation}
\label{Equ-3}
\mc{E}_\varphi (S)=\max_{F\in\mc{F}} \mc{E}_\varphi(F),
\qquad\text{where}\qquad
\mc{E}_\varphi (F) =E_\varphi(\lambda_{F,\varphi})=
\min_{\mu\in\mc{M}_1(F)}E_\varphi(\mu),
\end{equation}
where $\mc{M}_1(F)$ is the set of all probability Borel measures on $F$. 
If a solution $S$ of this extremal problem exists then under ``normal circumstances'' $S$ is an $S$-curve in the external field $\varphi$, see \cite{Rakhmanov:2012fk}.

\subsection{Critical measures}

\textit{Critical measures} are Borel measures on the plane with a vanishing local variation of their energy in the class of all smooth local variations with a prescribed set of fixed points. They are, therefore equilibrium distributions in a conducting plane with a number of insulating points of the external field. 

The zeros of the Heine--Stieltjes polynomials (see Section \ref{sec:HSpolyns}) are (discrete) critical measures. This observation lead in \cite{MR2770010} to a theorem on asymptotics of these polynomials in terms of continuous critical measures with a finite number of insulating points of the external field. 

It turns out that the equilibrium measures of compact sets with the $S$-property are critical; the reciprocal statement is that the critical measures may be interpreted as the equilibrium measures of $S$-curves in piece-wise constant external fields. Both notions, however, are defined in a somewhat different geometric settings, and it is in many ways convenient to distinguish and use both concepts in the study of many particular situations.

The idea of studying critical (stationary) measures has its origins in \cite{Gonchar:87}. Later it was used in \cite{Kamvissis2005,Rakhmanov94} in combination with the min--max ansatz, and systematically in \cite{MR2770010}, see also \cite{MR2647571,KuijlaarsSilva2015}. The formal definition is as follows: let   $\Omega\subset\mbb{C}$ be a domain,  $\mc{A}\subset\Omega$ be a finite set and  $\varphi$ a harmonic function in $\Omega\setminus \mc{A}$. Let $t\in \C$ and $h\in C^2(\Omega)$, satisfying $h\big|_\mc{A}\equiv 0$. Then for any (signed) Borel measure $\mu$ the mapping (``$\mc{A}$-variation'') $z\to z^t=z+th(z)$ defines the pull-back measure $\mu^t $, as well as the associated variation of the weighted energy,
\begin{equation}
\label{eq4.1}
D_hE_\varphi(\mu)=\lim_{t\to 0+}\frac{1}{t}(E_\varphi(\mu^t)-E_\varphi(\mu)),
\end{equation}
where $E_\varphi$ was introduced in \eqref{defWeightedEnergyCont}.
We say that $\mu$ is \textit{$(\mc{A},\varphi)$-critical} if for any $\mc{A}$-variation, such that the limit above exists, we have
\begin{equation}
\label{eq4.2}
D_hE_\varphi(\mu)=0;
\end{equation}
when $\varphi\equiv 0$, we write $\mc{A}$-critical instead of $(\mc{A},0)$-critical measure.

The relation between the critical measures and the $S$-property \eqref{SpropertyExtF} is very tight: every equilibrium measures with an $S$-property is critical, while the potential of any critical measure satisfies \eqref{SpropertyExtF}. However, in some occasion it is more convenient to analyze the larger set of critical measures.

As it was proved in \cite{MR2770010},  for any $(\mc{A},\varphi)$-critical measure $\mu$ we have
\begin{equation}\label{eq4.4}
R(z):=\left(\int\frac{d\mu(t)}{t-z}+\Phi'(z)\right)^2\in H(\Omega\setminus \mc{A}).
\end{equation}
This formula implies the following description of $\Gamma:=\supp(\mu)$: it consists of a finite number of critical or closed trajectories of the quadratic differential $-R(z)(dz)^2$ on $\C$ (moreover, all trajectories of this quadratic differential are either closed or critical, see \cite[Theorem 5.1]{MR2770010} or \cite[p.~333]{Gonchar:87}). Together with \eqref{eq4.4} this  yields the representation
\begin{equation}\label{eq4.5}
U^\mu(z)+\varphi(z)=-Re\int^z \sqrt{R(t)}\,dt+\const,\quad z\in\mbb{C}\setminus \Gamma.
\end{equation}
Finally, the $S$-property \eqref{SpropertyExtF} on $\Gamma$  and the formula for the density $d\mu(z)=\frac1{\pi}|\sqrt{R}\,dz|$ on any open arc of $\Gamma$ follow directly from \eqref{eq4.5}. In this way, function $R$ becomes the main parameter in the construction of a critical measure: if we know it (or can guess it correctly), then consider the problem solved.

\begin{example} \label{exampleSaffConf}
Let us apply the GRS theory to the following simple example: we want to study zeros asymptotics of the polynomials $Q_n\in \P_n$ defined by the orthogonality relation
\begin{equation}
\label{OrtExample}
\int_F Q_n(z)z^k w_n(z)dz=0,\quad k=0,1,\dots,n-1,
\end{equation}
where the varying (depending on $n$) weight function has the form 
\begin{equation}
\label{wExample}
w_n(z)  =  e^{- k n z} , \quad k\geq 0,
\end{equation}
and the integration goes along a Jordan arc $F$ connecting $a_1=-1- 2i$ and $a_2=\overline{a_1}=-1+2i$. Since   
$$
\varphi(z):=-\lim_{n\to \infty }\frac 1n \log {|w_n(z)|}=k \Re z
$$
uniformly on every compact subset of $\C$, the application of the GRS theory is reduced to finding a Jordan arc $S$, connecting $a_1$ and $a_2$, such that the equilibrium measure $\lambda=\lambda_{S,\varphi}$ has the $S$-property in the external field $\varphi$. 

For $k=0$, such an $S$ is just the vertical straight segment connecting both endpoints. Using \eqref{eq4.4} and the results of \cite{MR2770010} it is easy to show that for small values of $k>0$ (roughly speaking, for $0<k<0.664$), 
$$
R(z):=\left(\int\frac{d\lambda(t)}{t-z}+k\right)^2=\frac{(z-\beta)^2}{A(z)}, \quad A(z):=(z-a_1)(z-a_2), \quad \beta=-1+\frac{1}{k},
$$
valid a.e.~(with respect to the plane Lebesgue measure) on $\C$. In this case, $\supp\lambda$ is connected, and is the critical trajectory of the quadratic differential
\begin{equation}\label{quaddiffexample1}
-\frac{(z-\beta)^2}{A(z)}(dz)^2,
\end{equation}
connecting $a_2=-1\pm2i$. As an illustration of this case we depicted the zeros of $Q_{150}$ for $k=0.6$, see Figure \ref{FigExampleExp}, left\footnote{The procedure used to compute these zeros numerically is briefly explained in Section~\ref{sec:numerical experiments}.}.

\begin{figure}[h]
	\centering 
	\hspace{-5mm} \begin{tabular}{cc}
		\begin{overpic}[scale=0.65]{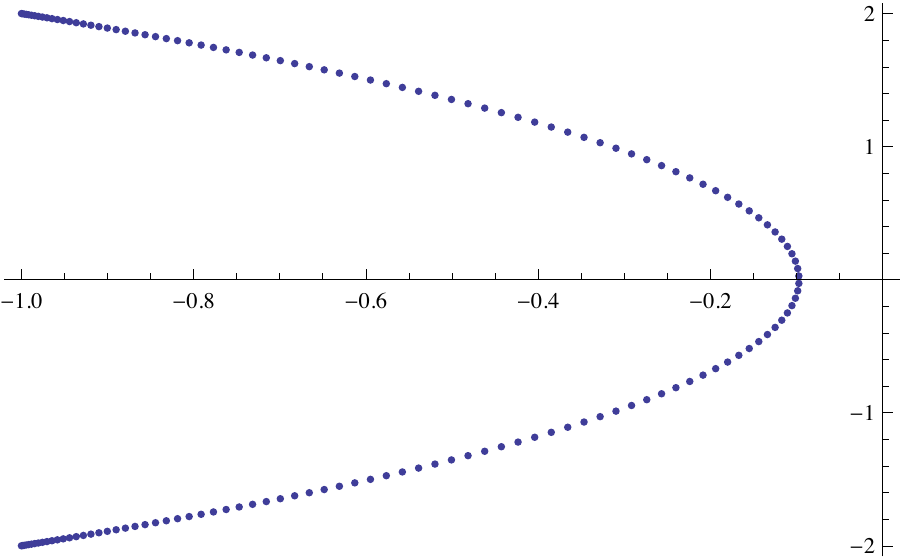}
		\end{overpic} & 
		\begin{overpic}[scale=0.65]{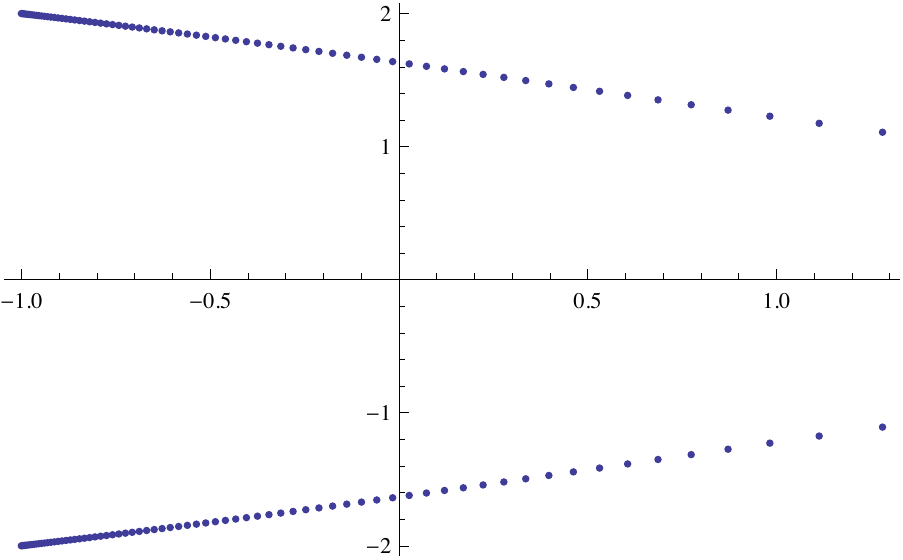}
	\end{overpic}
	\end{tabular}
	\caption{Zeros of $Q_{150}$, defined by \eqref{OrtExample}, for $k=0.4$ (left) and $k=0.8$ (right).}
	\label{FigExampleExp}
\end{figure}

At a critical value of $k=k^*\approx 0.664$ the double zero $\beta=-1+1/k$ hits the critical trajectory and splits (for $k>k^*$) into two simple zeros, that can be computed, so that now 
$$
R(z):=\left(\int\frac{d\lambda(t)}{t-z}+k\right)^2=\frac{(z-\alpha)(z-\overline{\alpha})}{A(z)}, \quad A(z):=(z-a_1)(z-a_2).
$$
In this situation $\supp(\lambda)$ is made of two real-symmetric open arcs, each connecting $a_j$ with $\alpha_j$ lying in the same half-plane, delimited by $\R$; they are the critical trajectory of the quadratic differential
$$
-\frac{(z-\alpha)(z-\overline{\alpha})}{A(z)}(dz)^2
$$
on $\C$. For illustration, see the zeros of $Q_{150}$ for $k=0.8$, see Figure \ref{FigExampleExp}, right.
\end{example}

Formula \eqref{eq4.5} and the fact that $\Gamma$ is comprised of arcs of trajectories of $-R(dz)^2$ on $\C$ show that $U^\lambda+\varphi$ is a constant on $\Gamma$. However, this constant is not necessarily the same on each connected component of $\Gamma$: this additional condition singles out the equilibrium measures with the $S$-property within the class of critical measures. Equivalently, the equilibrium measure can be identified by the validity of the variational condition \eqref{Equ-2}.

Recall that one of the motivations for the study of critical (instead of just equilibrium) measures is the characterization of the zero distributions of Heine--Stieltjes polynomials, see Section~\ref{sec:HSpolyns}. 
One of the main results in \cite{MR2770010} is the following. If we have a convergent sequence of Van Vleck polynomials $V_n(z)\to V(z)=z^{p-2}+\dots$, then for corresponding Heine--Stieltjes polynomials $Q_n$ we have
$$
 \mc{X}(Q_n)\ \overset{*}{\to}\ \mu,
$$
where $\mu$ is an $\mc{A}$-critical measure. Moreover, any $\mc{A}$-critical measure may be obtained this way.  
In electrostatic terms, $ \mc{X}(Q_n)$ is a discrete critical measure and the result simply means that a weak limit of a sequence of discrete critical measures is a (continuous) critical measure. In the case of the Heine--Stieltjes polynomials, $R$ in \eqref{eq4.4} is a rational function of the form $C/A^2$, with $A$ given by \eqref{defAandB}, and the polynomial $C$ determined by a system of nonlinear equations.

In the case of varying Jacobi polynomials, $P_n^{(\alpha_n, \beta_n)}$, with $\alpha_n$, $\beta_n$ satisfying
\begin{equation}\label{variationAB}
\lim_n \frac{\alpha_n}{n}=\mathfrak a\in \C,  \quad \lim_n \frac{\beta_n}{n}=\mathfrak b\in \C,
\end{equation} 
we can obtain the explicit expression for $R$ using the arguments of Section~\ref{sec:logpotential}. By the GRS theory, the problem of the weak-* asymptotics of the zeros of such polynomials boils down to the proof of the existence of a critical trajectory of the corresponding quadratic differential, joining the two zeros of $C$, and of the connectedness of its complement in $\C$, see  \cite{MR2124460,MR1805976,MR2142296,AMF:2015}, as well as Figure~\ref{fig:Withzeros}.

\begin{figure}[htb]
	\centering \begin{overpic}[scale=0.7]{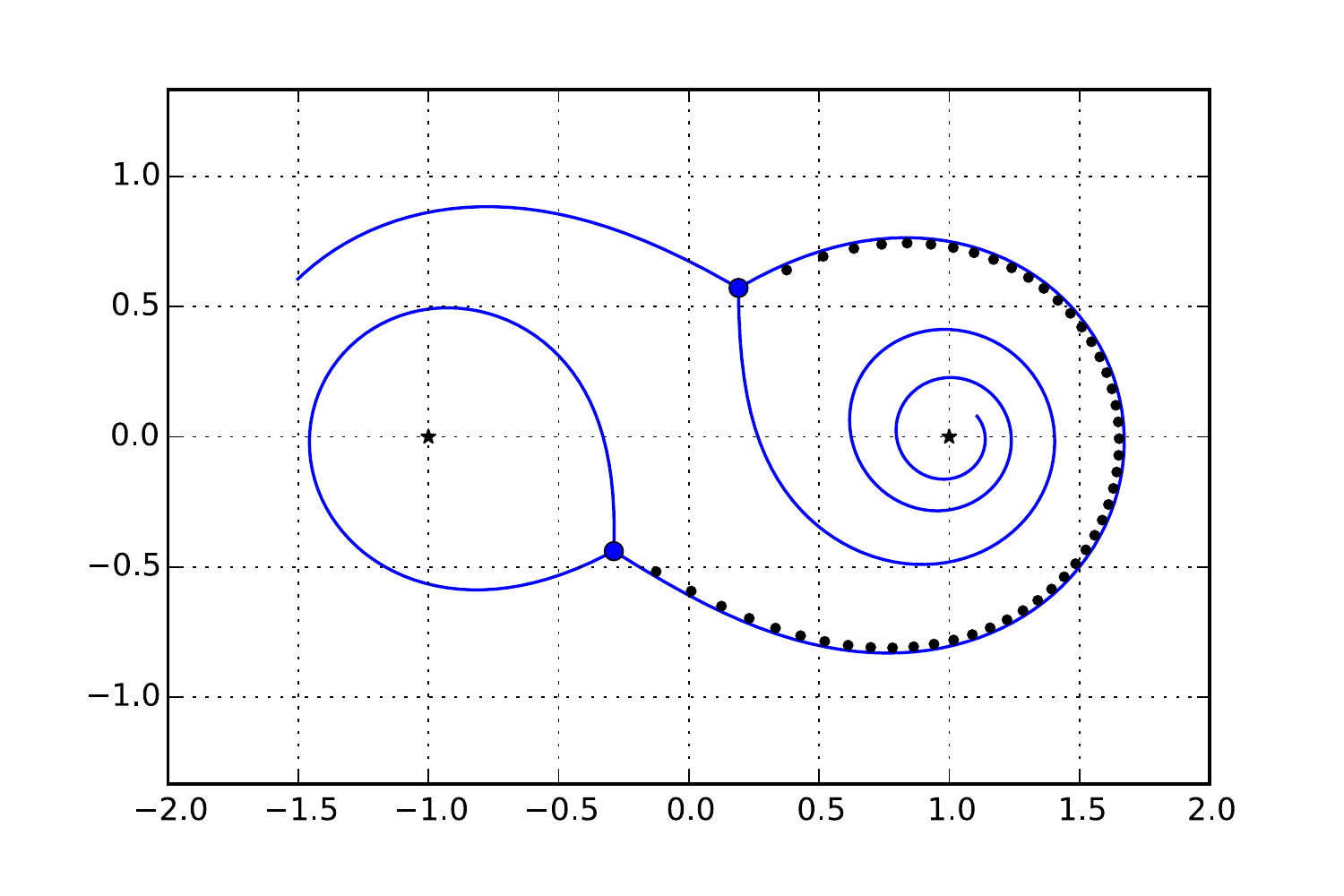}%
	\end{overpic}
	\caption{Critical trajectories  of  $-R(z)(dz)^2$, with $\mathfrak a=  -1.1 + 0.1i$ and $ \mathfrak b= 1$  in \eqref{variationAB}, and the zeros of the polynomial  from Figure~\ref{fig:onlyzeros} superimposed.}
	\label{fig:Withzeros}
\end{figure}

\section{Vector critical measures}\label{subsec6.7}

If we check the motivating examples from Section~\ref{sec:examples}, we will realize that at this point we only lack tools to explain the asymptotics of the zeros of the Hermite--Pad\'e polynomials (Section~\ref{sec:HPapprox}).
For this, we need to extend the notion of critical measures (and in particular, of equilibrium measures on sets with an $S$-property) to a vector case.

Assume we are given a vector of $p\geq 1$ nonnegative measures $\vec\mu=(\mu_1,\dots,\mu_p) $, compactly supported on the plane, a symmetric and positive-semidefinite interaction matrix $T=(\tau_{j,k})\in \R^{p\times p}$, and a vector of real-valued harmonic external fields $\vec\varphi=(\varphi_1,\dots,\varphi_p)$, $\varphi_j=\re \Phi_j$, $j=1,\dots,p$.  We consider the total (vector) energy functional of the form \cite{Gonchar:85}
\begin{equation}\label{energy_functional}
E(\vec \mu)=E_{\vec\varphi}(\vec\mu)=\sum_{j,k=1}^p \tau_{j,k} E(\mu_j,\mu_k)+2\sum_{j=1}^p\int \varphi_j d\mu_j,
\end{equation}
(compare with \eqref{defWeightedEnergyCont}), where
$$
E(\mu,\nu):=\iint\log\frac{1}{|x-y|}d\mu(x)d\nu(y)
$$
is the mutual logarithmic energy of two Borel measures $\mu$ and $\nu$.

Typical cases of the matrix $T$ for $p=2$ are
$$
\frac{1}{2}\,\begin{pmatrix}
2 & 1  \\ 1 & 2
\end{pmatrix} \quad \text{and} \quad \frac{1}{2}\,\begin{pmatrix}
2 & -1  \\ -1 & 2
\end{pmatrix},
$$
corresponding to the so-called
Angelesco and  Nikishin systems, respectively, see \cite{MR651757,MR1240781,MR1702555}.

As in the scalar situation, for $t\in \C$ and $h\in C^2(\C)$, denote 
by $\vec \mu^{\, t}$ the push-forward measure of $\vec \mu=(\mu_1,\dots,\mu_p)$ induced by the variation of the plane $z\mapsto h_t(z)=z+th(z)$, $z\in\C$. We say that  $\vec\mu$ is a \emph{critical vector 
	measure} (or a saddle point of the energy $E(\cdot)$) if  
\begin{equation}\label{variations_critical_measure}
D_hE_\varphi(\vec \mu)=\lim_{t\to 0+}\frac{1}{t}(E_{\vec \varphi}(\vec \mu^{\, t})-E_{\vec \varphi}(\vec \mu))=0,
\end{equation}
	for every function $h\in C^2(\C)$.

Usually, critical vector measures a sought within a class specified by their possible support and by some constraints on the size of each component. The vector \textit{equilibrium} problems deal with the minimizers of the energy functional  \eqref{energy_functional} over such a family of measures $\vec\mu$.

For instance, we can be given $p$ families of analytic curves $\mathcal T_j$,  so that $\supp \mu_j\in \mathcal T_j$, $j=1, \dots, p$, and additionally some constraints on the size of each component $\mu_j$ of $\vec \mu$. In a classical setting this means that we fix the values $m_1, \dots, m_p\geq 0$, such that $m_1+\dots+m_p=1$, and impose the condition
$$
|\mu_j|:=\int d\mu_j = m_j, \quad j=1, \dots, p,
$$
see e.g.~\cite{Gonchar:85,MR651757,MR2188362,MR2963452}. More recently new type of constraints have found applications in this type of problems, when the conditions are imposed on a linear combination of $|\mu_j|$. 

For instance, \cite{AMFSilva15} considers the case of $p=3$, with the interaction matrix
$$
T=\frac{1}{2}\,
\begin{pmatrix}
2 & 1 & 1 \\
1 & 2 & -1 \\
1 & -1 & 2
\end{pmatrix},
$$
polynomial external fields, and conditions
\begin{equation}
\label{massconstraint}
|\mu_1|+|\mu_2|=1,\quad |\mu_1|+|\mu_3|=\alpha,\quad |\mu_2|-|\mu_3|=1-\alpha,
\end{equation}
where  $\alpha\in [0,1]$ is  a parameter; see also  \cite{MR2475084} and  
\cite{MR2796829} for  $\alpha=1/2$ and its applications to Hermite-Padé approximants.
From the electrostatic point of view it means that we no longer fix the  total masses of each component $\mu_j$, but charges (of either sign) can ``flow'' from one component to another.
 
There is a natural generalization of the scalar $S$-property to the vector setting. Continuing with the results of the recent work \cite{AMFSilva15}, it was proved that under some additional natural assumptions on the critical vector measure $\vec \mu$,  if $\Sigma$ is an open analytic arc in $  \supp\mu_j$, for some $j\in \{1, 2, 3\}$, then there exists a constant $\omega=\omega(\Sigma)\in 
\R$ for which both the variational equation (\emph{equilibrium 
	condition})
\begin{equation}\label{variational_equations_critical_measures}
\sum_{k=1}^3a_{j,k}U^{\mu_k}(z)+ \varphi_{j}(z) =\omega,
\end{equation}
and the $S$-property
\begin{equation*}
\frac{\partial}{\partial n_+}\left(\sum_{k=1}^3 \tau_{jk}U^{\mu_k}(z)+ \varphi_{j}(z) \right)=\frac{\partial}{\partial 
	n_-}\left(\sum_{k=1}^3 \tau_{jk}U^{\mu_k}(z)+ \varphi_{j}(z) \right)
\end{equation*}
hold true a.e.~on $ \Sigma$, where again $n_\pm$ are the unit normal vectors to $\Sigma$, pointing in the opposite directions.

Clearly, the problem of existence of compact sets with the $S$-property for the vector problem are even more difficult than in the scalar case. The possibility to reduce this question to the global structure of critical trajectories of some quadratic differentials is the most valuable tool we have nowadays. This is given by higher-dimension analogues of the equation \eqref{eq4.4} for critical vector measures. Again, in the situation of \cite{AMFSilva15} it was proved that functions
 \begin{equation}\label{definition_xi_functions}
 \begin{aligned}
 \xi_1(z) & =\frac{\Phi_1'(z)}{3}+\frac{\Phi_2'(z)}{3}+C^{\mu_1}(z)+C^{\mu_2}(z),\\
 \xi_2(z) & =-\frac{\Phi_1'(z)}{3}-\frac{\Phi_3'(z)}{3}-C^{\mu_1}(z)-C^{\mu_3}(z),\\
 \xi_3(z) & =-\frac{\Phi_2'(z)}{3}+\frac{\Phi_3'(z)}{3}-C^{\mu_2}(z)+C^{\mu_3}(z),
 \end{aligned}
 \end{equation}
 are solutions of the cubic equation 
 \begin{equation}\label{spectral_curve_general}
 \xi^3- R(z)\xi+D(z)=0,
 \end{equation}
 where $R$ is a polynomial  and $D$ a rational function with poles of order at most $2$. This fact yields that  measures $\mu_1,\mu_2,\mu_3$ are supported on a finite union of analytic arcs, that are trajectories of a quadratic differential living on the Riemann surface of \eqref{spectral_curve_general}, and which is explicitly given on each sheet of this Riemann surface as $(\xi_i-\xi_j)^2(z)(dz)^2$.
The fact that the support of critical (or equilibrium) vector measures is described in terms of trajectories on a compact Riemann surface (and thus, what we actually see is their projection on the plane) explains the apparent geometric complexity of the limit set for the zeros of Hermite--Pad\'e polynomials, see Figure~\ref{Fig1HermitePade}.

\section{Orthogonality with respect to a sum of weights} \label{sec:6}

\subsection{General considerations}\label{subsec6.5}

Now we step into a new territory. 
It is well known that a multiplicative modification of an orthogonality measure is easier to handle (from the point of view of the asymptotic theory) than an additive one. However, this kind of problems appears very naturally. 
For instance, a large class of transfer functions corresponding to time-delay systems can be written as the ratio of two functions, each of them a polynomial combination of exponentials (see e.g.~\cite{MR1783294}). In the simplest case, let
$f(x)=\sum\limits^m_{k=1}c_k e^{-\lambda_k x}$ where $\lambda_k>0$ and $c_k\in\overline{\mbb C}$, and we want to consider the best rational approximations of $f$ on a given set $K$: $\rho_n=\min\limits_{\Gamma\in\mbb R}\max\limits_{x\in  K}|f(x)-r(x)|$. This is a direct generalization of the central problem solved in \cite{Gonchar:87}; in particular, it can be reduced to the asymptotics of polynomials of non-Hermitian orthogonality with respect to a varying weight containing $f$, i.e. ~a \emph{sum} of exponentials.

In the rest of the paper we will consider mostly the simplest model situation, which still contains the most essential points, when polynomials $Q_n(z)=z^n+\dotsb$ are defined by orthogonality relations 
\begin{equation}
\label{Ort}
\int_F Q_n(z)z^k w_n(z)dz=0,\quad k=0,1,\dots,n-1,
\end{equation}
and the varying (depending on $n$) weight function has the form 
\begin{equation}
\label{w}
w_n(z)   =   e^{-k_1 n z} + e^{-k_2 n z}, \quad k_1, k_2 \in \C,
\end{equation}
where the integration goes along a curve
\begin{equation}
\label{C}
F \in \mc F := \{\text{rectifiable curve }   F \in \C  \text{ connecting } a_1  \text{ and }   a_2\},
\end{equation}
with $a_1, a_2 \in \C$ being two distinct   points fixed on the plane. 

Let us get ahead of ourselves and announce the status of the problem of the asymptotic zero distribution of polynomials $Q_n$: it  depends on the values of the parameters $a_1, a_2, k_1, k_2$. In some range the solution is essentially known; there is another parameters domain where a modification of known methods make it possible to prove a 
theorem on the zero distribution by generalizing known results. In particular a new equilibrium problem has to be introduced  in order to present these generalizations. Finally, there is a ``gray area'' in the space of parameters, where we can only show some results of numerical experiments and to state partial conjectures explaining those 
results. Boundaries between above mentioned domains in the space of parameters are not explicitly known yet; in some cases we  have partial results which are mentioned below. 

The integral in \eqref{Ort} resembles, at least formally, the orthogonality conditions in \eqref{eq1.7}, so it is natural to  analyze first the applicability of the GRS theory, explained in Section~\ref{sec:GRS}.
The starting assumption is that the weights $w_n(z) $ in \eqref{Ort} must be analytic in a domain $\Omega$ that contains points $a_1, a_2$, where the following limit exists: 
\begin{equation}
\label{Conv}
\varphi(z):=- \lim_{n\to \infty }\frac 1n \log {|w_n(z)|}   , \quad z \in \Omega.
\end{equation} 
Notice that in the original paper \cite{Gonchar:87} the convergence in the above formula was assumed to be uniform on compacts in $\Omega$, and consequently the function $\varphi$ was harmonic in $\Omega$. These assumption need to be relaxed for our purposes. For instance, if in the case \eqref{w} we assume $k_1> k_2$, then
\begin{equation} \label{phiparticularcase}
\varphi(z)= \min \left\{k_1 \Re z, k_2 \Re z \right\}=
\begin{cases}
k_2 \Re z, & \text{if } \Re z \ge 0,\\
k_1 \Re z, & \text{if } \Re z <0.
\end{cases}
\end{equation}
This function is continuous and piece wise harmonic in $\C$ (harmonic in the left $\C_-:=\{z\in \C:\, \Re z <0\}$ and right $\C_+:=\{z\in \C:\, \Re z >0\}$  half-planes). Moreover, it is superharmonic and continuous in the whole plane. Convergence in \eqref{Conv} is uniform on compact sets in the complement to the imaginary axis (notice that $\C_- \cup  \C_+$ is an open set). It is also easy to verify that for an arbitrary compact set in $\C$ the convergence is still uniform, but from above: for any compact $K \subset \C$ and any $\varepsilon >0$ there is $N(\varepsilon) \in \N$ such  that 
$$
\frac 1n \log {|w_n(z)|}  <   - \varphi(z) + \varepsilon , \quad z \in K \quad \text{for} \quad n \geq N(\varepsilon).
$$

In the light of this example it is natural to try to consider the analogue of the GRS theory under the assumptions that the external field $\varphi$ is continuous and piece-wise harmonic in $\Omega$, and convergence in \eqref{Conv} is locally uniform in the open sets where $\varphi$ is harmonic, and locally uniform from above in $\C$.

\subsection{Numerical experiments}\label{sec:numerical experiments}

Before continuing, and in order to gain some intuition, let us present the result of a few numerical experiments, finding the zeros of $Q_n$ given by \eqref{Ort}--\eqref{C}. These calculations can be carried out in a rather stable way using classical approach via moment determinants, see \cite[Chapter 2]{Gautschi:96}. Indeed, 
$$
\int  z^j e^{-s z}dz=(-1)^j s^{-j-1} \Gamma (j+1,-s z)+\const, \quad j\geq 0,
$$
where $\Gamma(a,z)$ is the incomplete gamma function, so that 
the moments
$$
\widehat \mu_j(s):=\int_F z^j e^{-s z}dz , \quad j\geq 0,
$$
can be either  evaluated  in terms of special functions or obtained from the recurrence (for $s\neq 0$)
$$
\widehat \mu_j(s)=\frac{a_1^{j}e^{-sa_1}-a_2^{j}e^{-sa_2}+ j \widehat \mu_{j-1}(s)}{s}, \qquad \widehat \mu_{-1}(s)=0.
$$
In practice, it might be convenient to apply this recurrence directly to the moments 
$$
 \mu_j^{(n)}:=\int_F z^j w_n(z) dz , \quad j\geq 0,
$$
in order to avoid cancellation errors.
Using the identities from \cite[Theorem 2.2]{Gautschi:96} we find the coefficients of the recurrence relation, satisfied by the polynomials, orthogonal with respect to $w_n$, and calculate their zeros as eigenvalues of the corresponding (complex) Jacobi matrix. All experiments were performed with Mathematica version~10.3, using an extended precision.
  
Let us discuss some of the pictures presented here. 

\begin{figure}[h]
	\centering 
	\hspace{-3mm} \begin{tabular}{c@{\hskip 3mm}c}
		\begin{overpic}[scale=0.65]{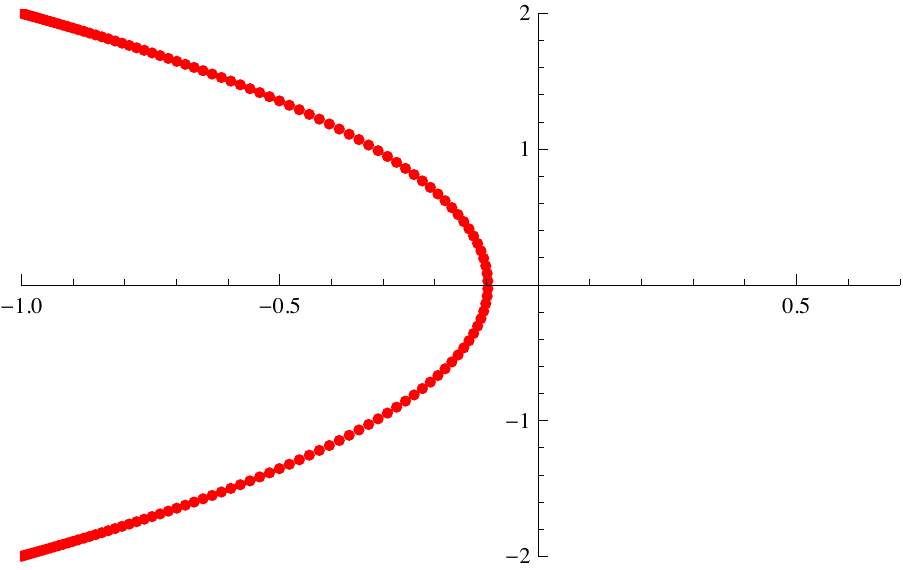}
		\end{overpic} & 
		\begin{overpic}[scale=0.65]{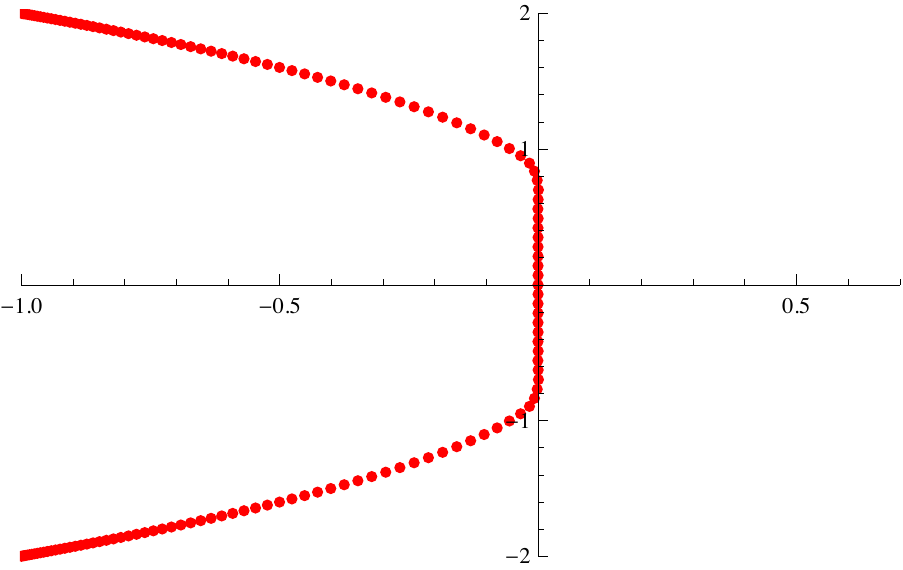}
		\end{overpic} \\
		\begin{overpic}[scale=0.65]{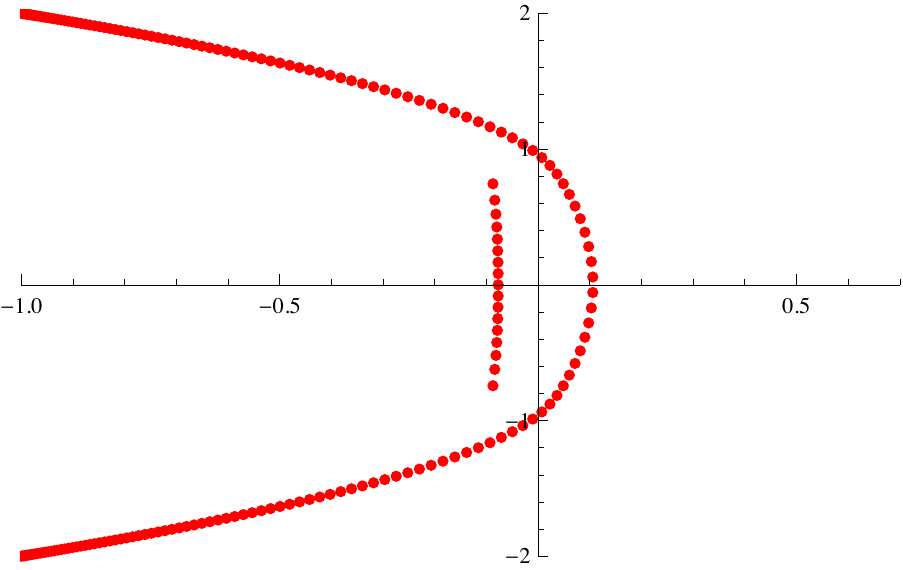}
		\end{overpic} & 
		\begin{overpic}[scale=0.65]{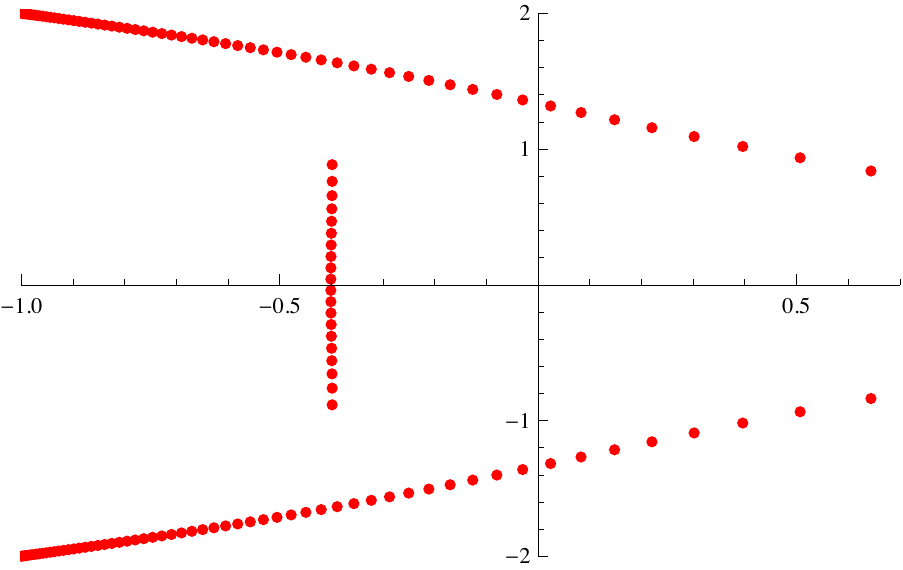}
		\end{overpic}
	\end{tabular}
	\caption{Zeros of $Q_{150}$ for $w_n(z) = 1+ e^{- k n z}$  with $a_1=-1- 2i=\overline{a_2}$ and and $k=0.4$ (top left), $k=0.6$ (top right), $k=0.66$ (bottom left) and $k=0.8$ (bottom right).}
	\label{FigExampleSaffConf}
\end{figure}

To start with, we consider the case when both $a_1$ and $a_2$ belong to the same half plane by modifying the orthogonality weight from Example~\ref{exampleSaffConf} and taking
\begin{equation}
\label{wExampleMod}
w_n(z)  = 1+ e^{- k n z} , \quad k\geq 0,
\end{equation}
with the integration in \eqref{Ort} along a Jordan arc $F$ connecting $a_1=-1- 2i$ and $a_2=-1+2i$, see Figure~\ref{FigExampleSaffConf}. The upper left picture is clear and should be compared with Figure~\ref{FigExampleExp}, left: all zeros lie in the left half-plane, where the resulting external field \eqref{phiparticularcase}, namely
\begin{equation}\label{phiexample2}
\varphi(z)=  
\begin{cases}
0, & \text{if } \Re z \ge 0,\\
k \Re z, & \text{if } \Re z <0,
\end{cases}
\end{equation}
is harmonic. The asymptotic distribution of zeros is $\lambda_{S,\varphi}$, and $\supp (\lambda_{S,\varphi})$ is the trajectory of the quadratic differential \eqref{quaddiffexample1}, joining $a_1$ and $\overline{a_1}$. 

Recall that as $k>0$ increases, this trajectory approaches the imaginary axis, until for certain $k=k_*\approx 0.43$ it touches $i\R$ at a single point.  At this point the $S$-property is no longer valid: it does hold for the external field
$k \Re z$, but not for \eqref{phiexample2}! What happens next is shown in Figure~\ref{FigExampleSaffConf}, top right. 

Now, recall that in the standard GRS theory the max-min property and the symmetry or $S$-property of the support are basically equivalent. For the piece-wise harmonic $\varphi$ this is no longer the case. Indeed, let $\Gamma$ be the critical trajectory of \eqref{quaddiffexample1} joining $a_1$ and $\overline{a_1}$ and that crosses the imaginary axis (so $k> k_*$), and let $\widehat \Gamma$ be the continuum obtained by replacing $\Gamma \cap \{ \Re z\geq 0 \}$ with the straight vertical segment joining the two points of intersection (see Figure~\ref{maxmin}).  Then
$$
\mc{E}_\varphi (\Gamma) \leq \mc{E}_\varphi (\widehat\Gamma)
$$
(see the definition in \eqref{Equ-3}). In other words, at least in the situation of Figure~\ref{FigExampleSaffConf}, top right, the zeros of $Q_n$ ``prefer'' to follow the max-min configuration instead of the $S$-curve\footnote{As a referee pointed out, the configuration of Figure~\ref{FigExampleSaffConf}, top right, could be formally considered as a limit case of what we describe in Theorem~\ref{thm2} below, where the $S$-property does play a role. Nevertheless, the situation is completely different here due to the existence of a string of zeros of the weight \eqref{wExampleMod} along the imaginary axis. In fact, zeros of $Q_n$, at least visually, are equally spaced on $i\R$, following the distribution of zeros of $w_n$ and not an equilibrium measure on an interval of $i\R$.}.

\begin{figure}[h]
	\centering 
	\hspace{-3mm} \begin{tabular}{c@{\hskip 15mm}c}
		\begin{overpic}[scale=0.9]{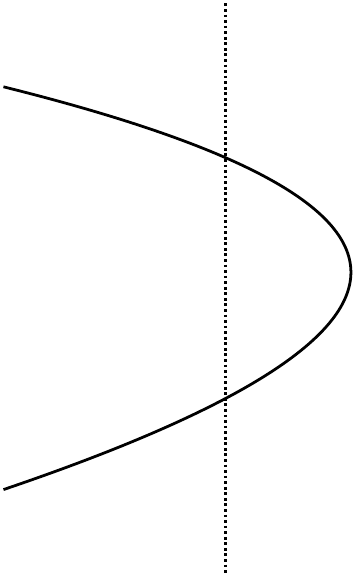}
						\put(5,32){ $\Gamma$}
\put(02,135){\small $\varphi(z)= k \Re z$}
\put(67,135){\small $\varphi(z)= 0$}						
		\end{overpic} & 
		\begin{overpic}[scale=0.9]{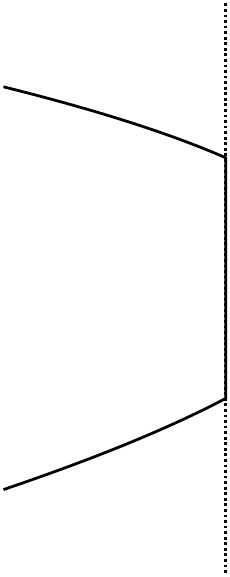}
\put(5,32){ $\widehat \Gamma$}
\put(02,135){\small $\varphi(z)= k \Re z$}
\put(67,135){\small $\varphi(z)= 0$}
		\end{overpic} 
	\end{tabular}
	\caption{Trajectory $\Gamma$ (left) and its projection $\widehat \Gamma$ onto $\C_-$.}
	\label{maxmin}
\end{figure}
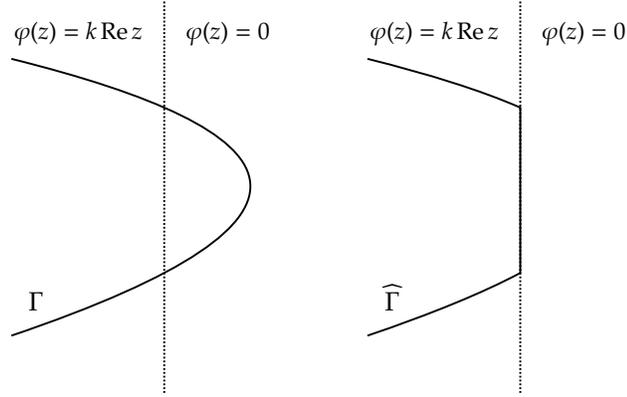

This does not last too long: the possibility of a symmetry is recovered on
Figure~\ref{FigExampleSaffConf}, bottom row, which should be compared with Figure~\ref{FigExampleExp}. We observe the appearance of a new connected component of the zero distribution, whose contribution allows for the $S$-property of the support, but now in a ``vector'' equilibrium setting, see the discussion next.

\medskip

\begin{figure}[h]
	\centering 
	\hspace{-3mm} \begin{tabular}{c@{\hskip 3mm}c}
		\begin{overpic}[scale=0.65]{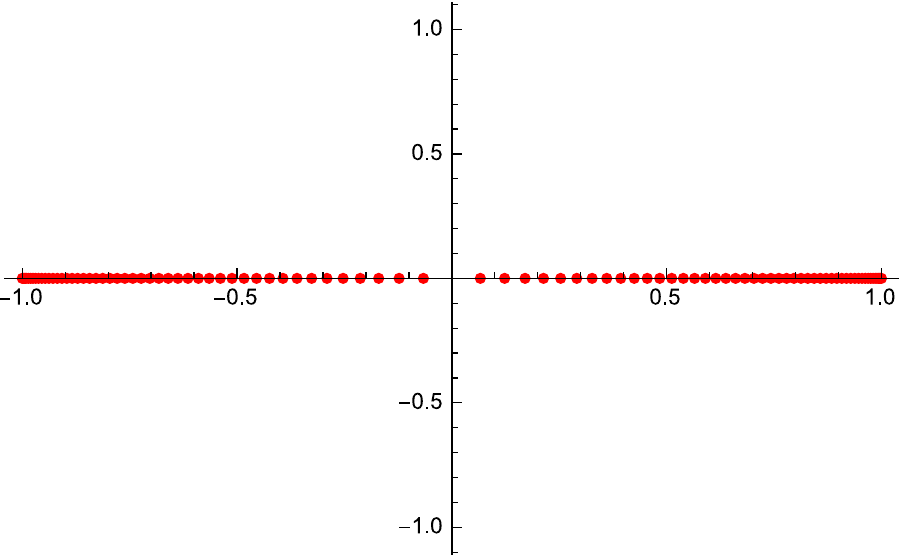}
		\end{overpic} & 
		\begin{overpic}[scale=0.65]{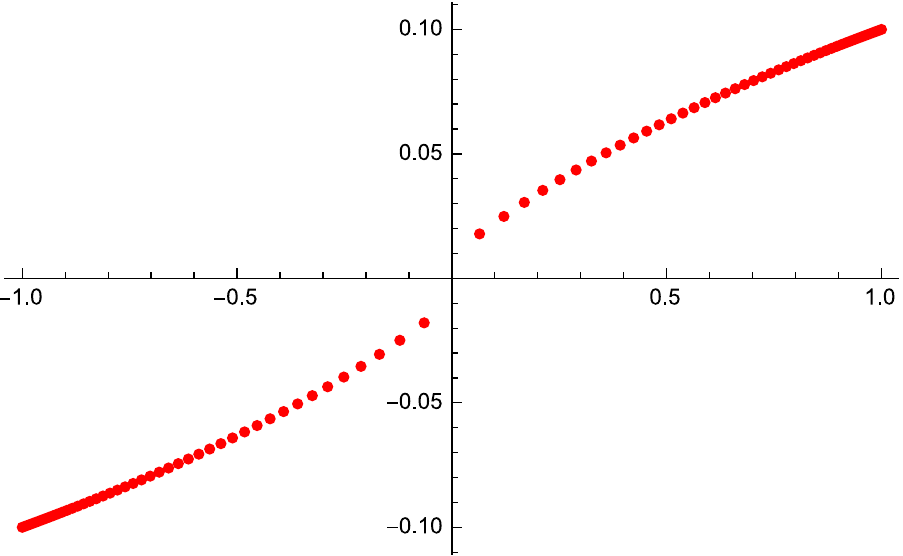}
		\end{overpic} 
	\end{tabular}
	\caption{Zeros of $Q_{100}$ for $w_n$ in \eqref{cosh} with $k=0.7$, for $a_2=-a_1=1$ (left) and $a_2=-a_1=1+0.1 i$ (right).}
	\label{FigCosReal}
\end{figure}

In the second set of experiments we consider the case  when $k_2=-k_1=k$, that is, 
$$
w_n(z)=e^{-k n z} + e^{k n z}=2 \cosh(k n z).
$$
In other words, $Q_{n}$ are now determined by the orthogonality conditions
\begin{equation}\label{cosh}
\int_{a_1}^{a_2} Q_{n}(z)z^k  \cosh(k n z) dz= 0,\quad k=0,1,\dots,n-1.
\end{equation}
In Figure~\ref{FigCosReal} we represent the zeros of $Q_{100}$ for $a_1=-a_2$; in particular, we take the totally real case $a_2=1$ (Figure \ref{FigCosReal}, left), and a slightly perturbed case $a_2=1+0.1 i$ (Figure \ref{FigCosReal}, right). Notice that all zeros apparently stay either in the left or right half-plane, which corresponds to the situation described by Theorem~\ref{thm2} below.

\begin{figure}[h]
	\centering 
	\hspace{-3mm} \begin{tabular}{c@{\hskip 3mm}c}
		\begin{overpic}[scale=0.65]{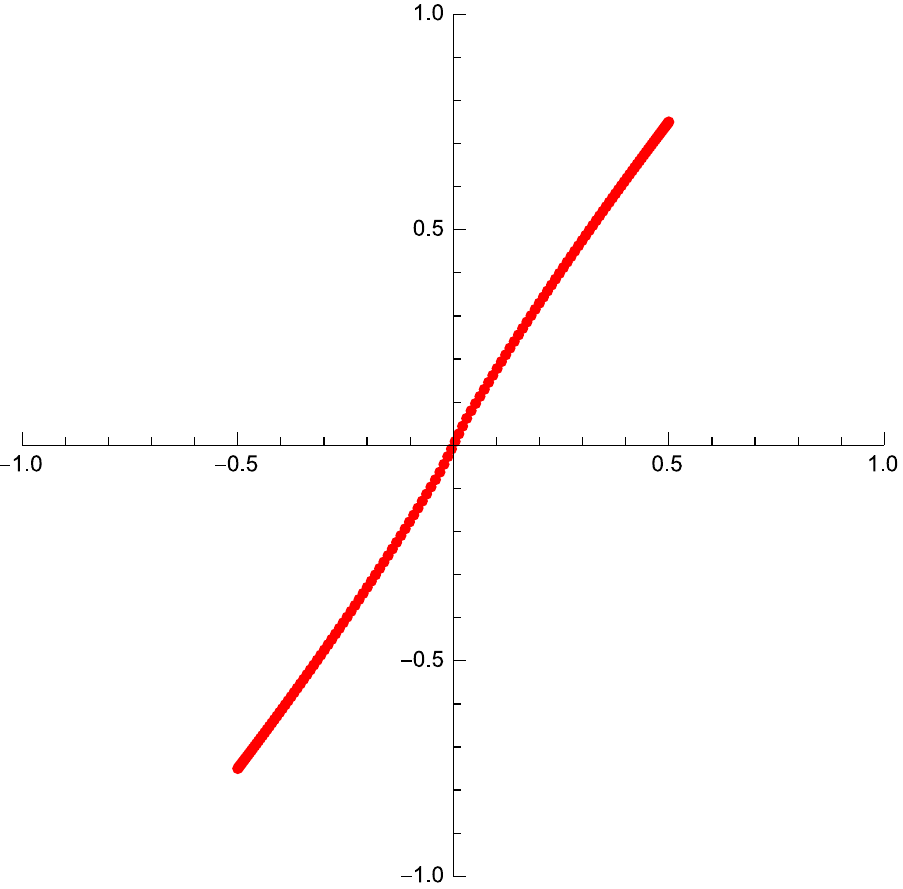}
		\end{overpic} & 
		\begin{overpic}[scale=0.65]{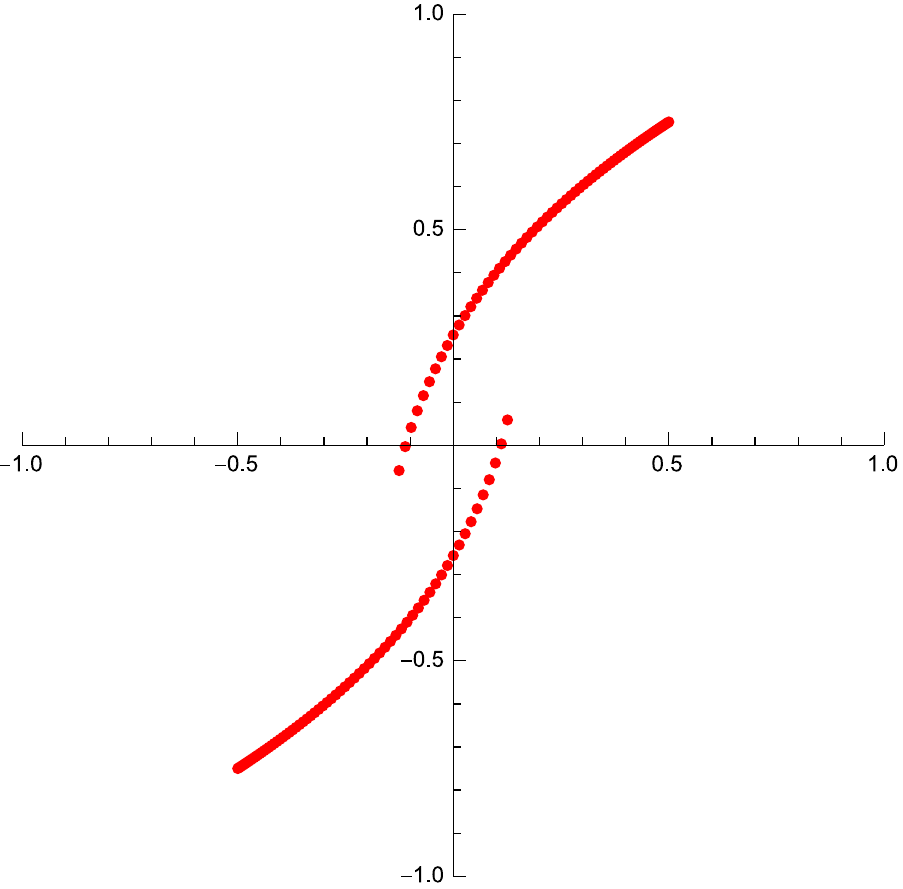}
		\end{overpic} 
	\end{tabular}
	\caption{Zeros of $Q_{150}$ for $w_n$ in \eqref{cosh} for $a_2=-a_1=0.5+0.75 i$, with $k=0.13$,  (left) and $k=1$ (right).}
	\label{FigAssym1}
\end{figure}

However, different values of the parameters $a_1$ and $k$ produce different pictures. 
In Figure~\ref{FigAssym1} we represent the zeros of $Q_{150}$ again for $a_1=-a_2$. But now  $a_2=-a_1=0.5+0.75 i$ with $k=0.13$ (left) and $a_2=-a_1=0.5+0.75 i$ with $k=1$ (right). It looks like the limit distribution  of the zeros of $Q_n$ can be now either a continuum, cutting the imaginary axis at a single point, or it can split into disjoint arcs, each or some of them also crossing the imaginary axis. We will give a partial answer (at least, a conjecture) for this case in our discussion in the next section.

\begin{figure}[h]
	\centering 
	\hspace{-3mm} \begin{tabular}{c@{\hskip 3mm}c}
		\begin{overpic}[scale=0.65]{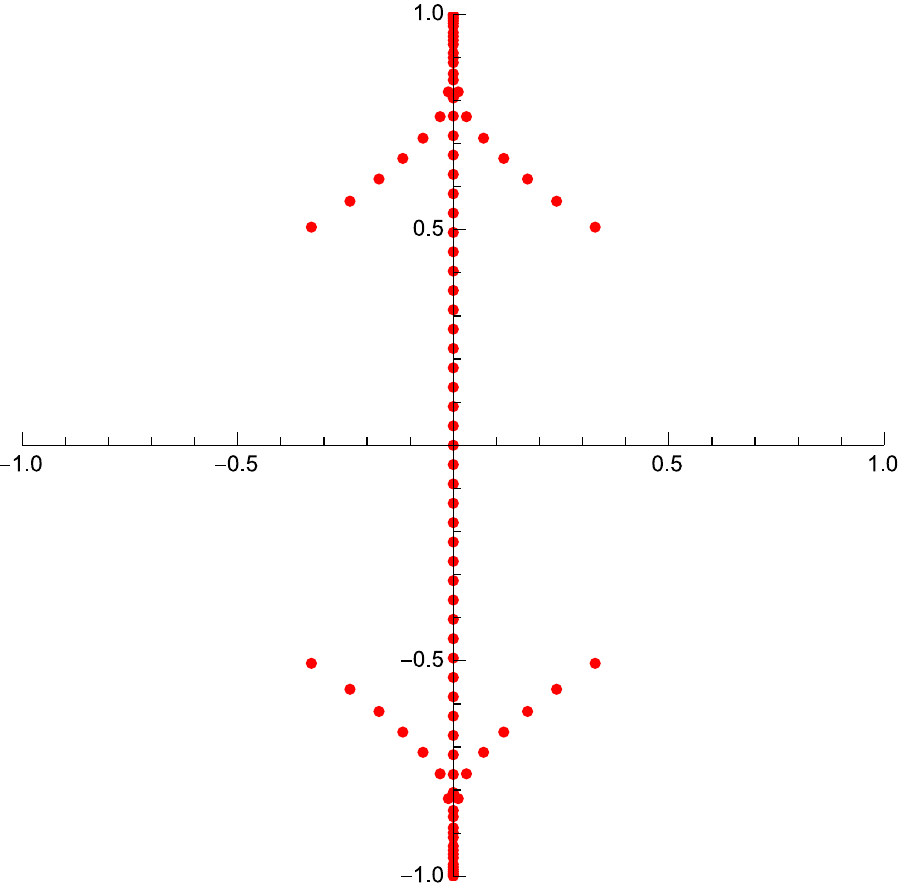}
		\end{overpic} & 
		\begin{overpic}[scale=0.65]{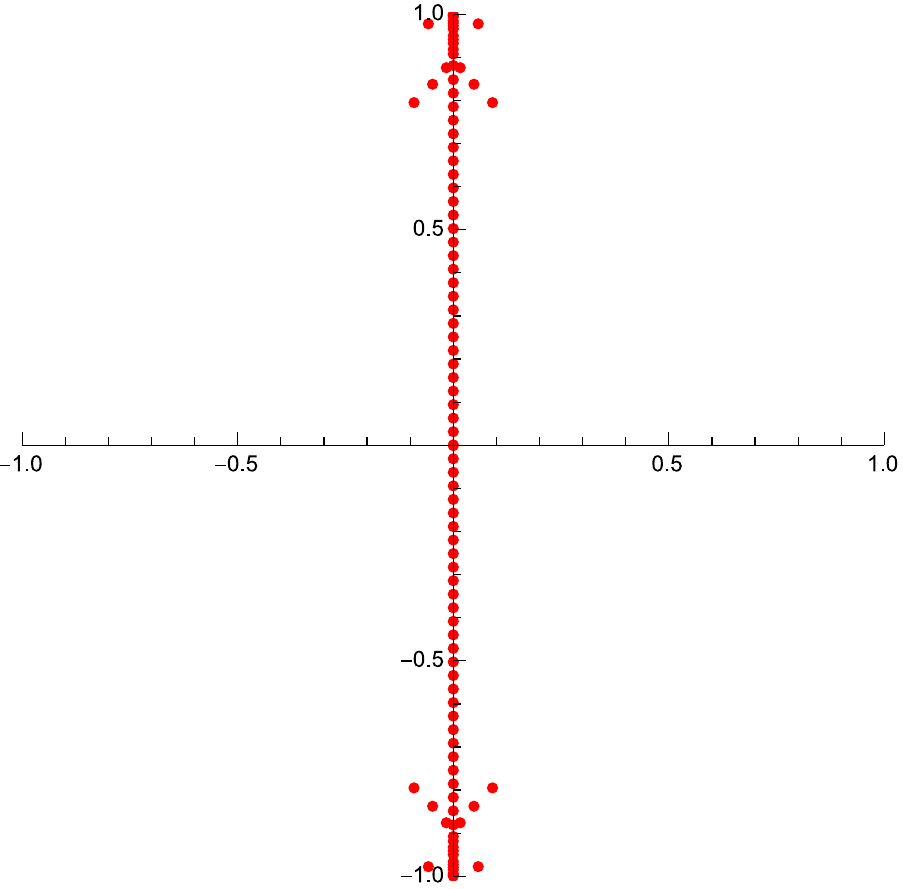}
		\end{overpic} \\
		\begin{overpic}[scale=0.65]{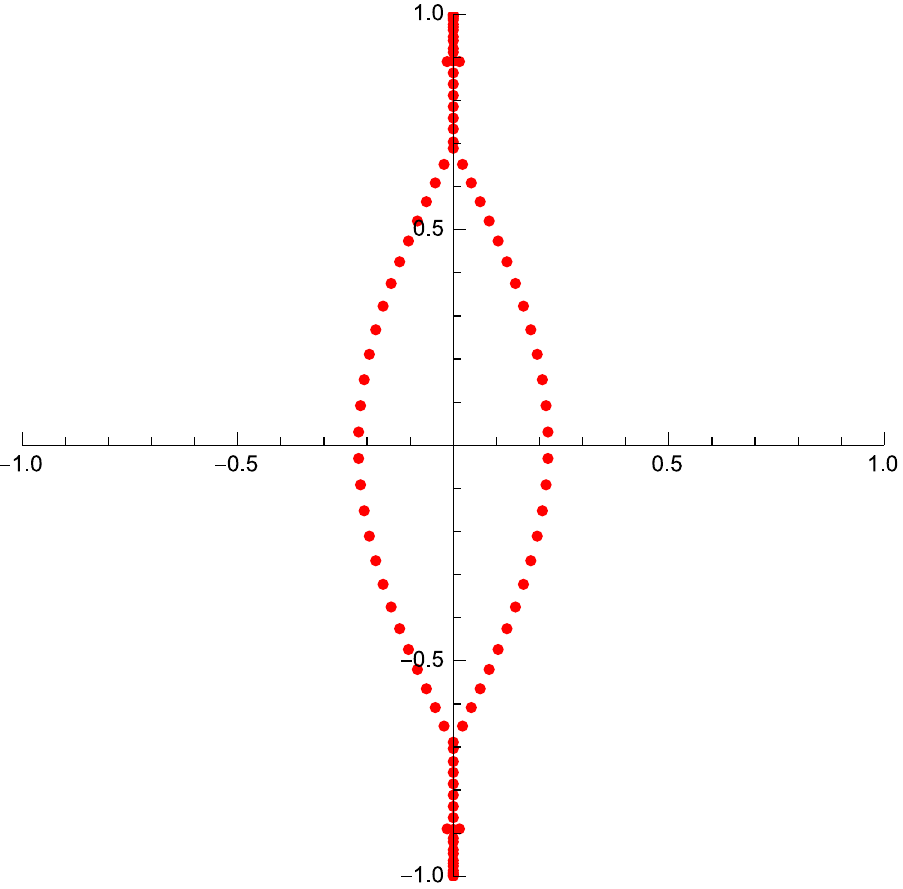}
		\end{overpic} & 
		\begin{overpic}[scale=0.65]{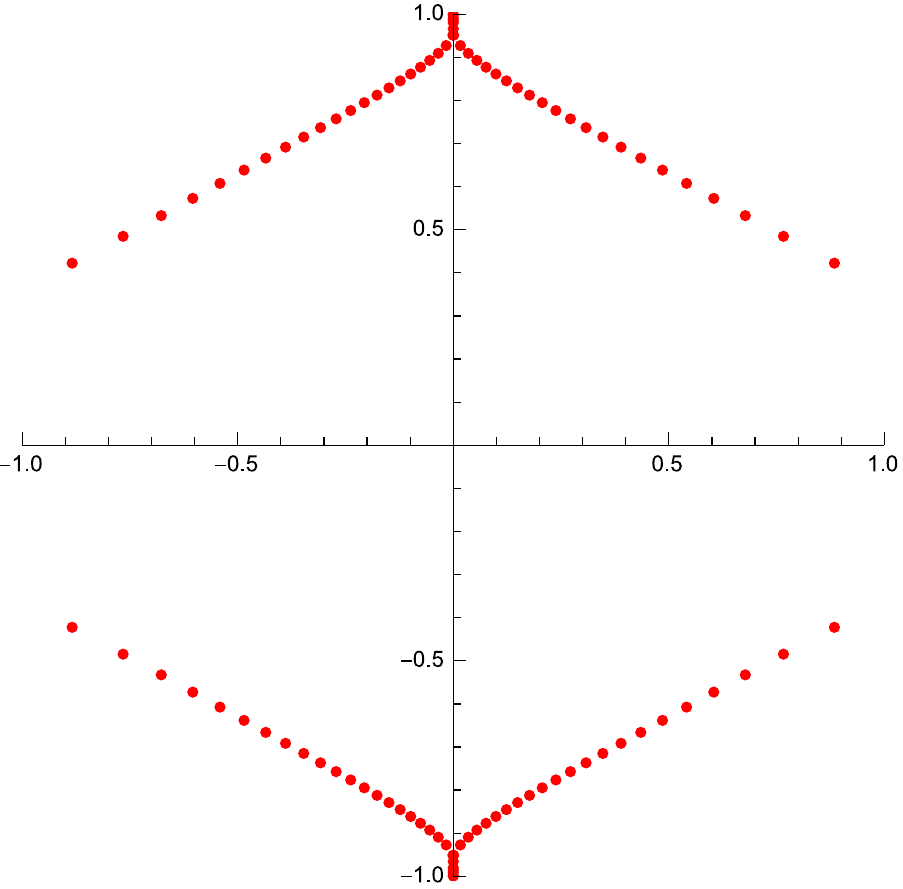}
		\end{overpic}
	\end{tabular}
	\caption{Zeros of $Q_{100}$ for $w_n$ in \eqref{cosh} with $a_2=-a_1=i$ and $k=0.7$ (top left), $k=1$ (top right), $k=1.2$ (bottom left) and $k=1.5$ (bottom right).}
	\label{FigCosImag}
\end{figure}

Let us finally reach the complete symmetry by analyzing the orthogonality \eqref{cosh}, with $a_2=i=-a_1$, Figure~\ref{FigCosImag}. Notice that now the path of integration can go entirely along the imaginary axis, precisely where the weight $w_n(z)=2 \cosh(k n z)$ has its zeros. A similar situation has been analyzed in \cite{Arno_Bessel14} for a semi-infinite interval.

We observe that the zeros of $Q_n$ in Figure~\ref{FigCosImag} can be split basically into three groups, two of them on $i\R$, and the rest on curves in the complex plane. Each curve belongs entirely (with a possible exception of its end points) to a half-plane $\C_\pm$. There are also to points $b_1=-ib=b_2$, $0<b<1$, such that the vertical segment $[-ib, ib]$ either contains no zeros of $Q_n$, or they are equally spaced, in concordance with the zeros of the weight $w_n$ there. The complement of this vertical segment in $[-i,i]$ also contains zeros of $Q_n$, but now their distribution is quite different, with a clear blow up in the density toward the endpoints $\pm i$. The whole picture resembles (but we do not claim at this point anything further beyond a formal similarity) the constrained equilibrium \cite{Rakhmanov:96,Dragnev97,Baik02,MR2283089}, with its void, bands and saturated regions, with the ``constraint'' created by the zeros of the weight.

Pictures in Figure~\ref{FigCosImag} may look puzzling also in the light of the electrostatic interpretation of the zeros of $Q_n$. Indeed, the cross section of the external field \eqref{phiparticularcase} along any line parallel to $\R$ is  convex upward (i.e., it looks like the graph of $-|x|$ on $\R$). Such a field should push the zeros away from the imaginary axis, while what we see is more appropriate of a field of the form $|x|$ on $\R$. Instead of relegating again the explanation of this phenomenon to the next section, let us point out here the simple but key fact for the orthogonality \eqref{Ort}--\eqref{C}, which allows us to ``swap'' the components of the external field \eqref{phiparticularcase}.

Indeed, the orthogonality condition in \eqref{Ort} can be equivalently rewritten as
$$
\int_{F_1} Q_n(z)z^k e^{-k_1 n z} dz+ \int_{F_2} Q_n(z)z^k    e^{-k_2 n z}dz=0,\quad k=0,1,\dots,n-1,
$$
where $F_1$, $F_2$ are \emph{any} two curves from the class $ \mc F$ defined in \eqref{C}. Alternatively, we could state it as
$$
\oint_{C} Q_n(z)z^k \widehat w_n(z) dz=0,\quad k=0,1,\dots,n-1,
$$
where $C$ is an oriented close Jordan contour on $\C$ passing through $a_1$ and $a_2$, and where $\widehat w_n(z)$ is defined either as $e^{-k_1 n z}$ or $-e^{-k_2 n z}$ on two distinct connected components of $C\setminus \{ a_1, a_2\}$.

We see now that we can actually ``pull'' the curve with $e^{k n z}$ to the right half-plane (assuming $k>0$), and the curve with $e^{-k n z}$ to the left one, so that the corresponding external field becomes just $-\varphi$, with $\varphi$ in \eqref{phiparticularcase}! So, we have traded a field with the convex upward cross section  along $\R$ for a convex downward one.

The trick we have just performed is an expression of the fact that the orthogonality with respect to a sum of weights can be associated to a vector equilibrium problem, similar to  those described in Section~\ref{subsec6.7}. For more details, see the discussion below.

\subsection{GRS theory for the for piece-wise harmonic external fields} \label{GRSpiecewise}
  
Recall that we are interested in a generalization of the  GRS theory under the assumptions that the external field $\varphi$ is continuous and piece-wise harmonic in $\Omega$, 
and convergence in \eqref{Conv} is as described at the end of Subsection~\ref{subsec6.5}. It turns out that such a generalization is rather straightforward 
when the support $\Gamma_\lambda$ of the equilibrium measure $\lambda_{\varphi,S}$ for the $S$-curve $S$ in the field $\varphi$ does not intersect the boundary  of the domains where $\varphi$ is harmonic.
  
  To be more precise, we assume that:
  \begin{enumerate}
  	\item[(i)]  the sequence $w_n \in H(\Omega)$ for a domain $\Omega$, 
  	\item[(ii)] there is a finite number of disjoint sub-domains $\Omega_1,\dots, \Omega_s$ of $  \Omega$ such that convergence in \eqref{Conv} is uniform on compact subsets of each $\Omega_j$ and locally uniform form above in $\Omega$,
  	\item[(iii)] $\varphi$  is continuous in $\Omega$.
  \end{enumerate}
 Observe that by our assumptions, $\varphi$  is also harmonic in each $\Omega_j $.   
 
 Given a finite set $\mathcal{A}=\left\{a_1,\dotsc,a_p\right\}$ of $p\ge2$ distinct points in $\Omega$, let $\mc{T}_{\mc{A}}$ denote the family of continua made of a finite union of oriented rectifiable arcs and containing $\mc{A}$.
  \begin{thm}\label{thm2} 
  	Under the assumptions above, let polynomials $Q_n(z)=z^n+\dotsb$ be defined by the orthogonality relations \eqref{Ort},	where the integration goes along $F\in \mc{T}_{\mc{A}}$. 
  	
  	If there exists $S\in \mc{T}_{\mc{A}}$ with the $S$-property in the external field $\varphi$ such that 
  	$$
  	\supp(\lambda_{S,\varphi})\subset \bigcup_{j=1}^s \Omega_j \quad \text{and} \quad \C\setminus \supp(\lambda_{S,\varphi}) \text{ is connected,}
  	$$ 
  	then $\chi(Q_n)\overset{*}{ \longrightarrow}\lambda_{S,\varphi}$.
  \end{thm}
  Recall that $\lambda_{S,\varphi}$ is the equilibrium measure on $S$ in the external field $\varphi$ and satisfying the symmetry property \eqref{SpropertyExtF}.
  
  The proof of the original version of the theorem in \cite{Gonchar:87} is directly applicable to the slightly more general conditions of Theorem~\ref{thm2}, without any essential modification. In fact, it is possible to push the situation even further and prove the assertion when $\supp (\lambda)$ cuts the set of discontinuities of $\varphi'$ in a finite number of points.
  
  Let us summarize the most important assumptions of Theorem~\ref{thm2}:
  \begin{enumerate}
  	\item[(a)] the existence of an $S$ curve, and
    \item[(b)] the fact that $\varphi$ is harmonic in a neighborhood of  the support of $\lambda_{S,\varphi}$.
  \end{enumerate}
  
  It is not an easy problem to find general conditions in terms of the field $\varphi$ (or equivalently, of the weights $w_n$) under which both (a) and (b) are automatically satisfied. In our model case \eqref{Ort}--\eqref{C} it is possible at least to indicate a system of equations (see Theorem~\ref{thm-3} ) which could lead to a complete or partial solution of the problem.

  \subsection{Pseudo-vector equilibrium and critical measures}
\label{sec:pseudo}

Recall the definition of the vector equilibrium from Section~\ref{subsec6.7}, which involves a vector of $p\geq 1$ nonnegative measures $\vec\mu=(\mu_1,\dots,\mu_p) $, a symmetric and positive-semidefinite interaction matrix $T=(\tau_{j,k})\in \R^{p\times p}$, and a vector of external fields $\vec\varphi=(\varphi_1,\dots,\varphi_p)$, along with a set of constraints on the size of the components of $\vec \mu$.  The equilibrium measure is the minimizer of the total vector energy functional \eqref{energy_functional}.

We call such a problem \emph{pseudo-vector} if all entries of the interaction matrix $T$ are just $1$'s. The terminology comes from the observation that in this case the total energy is just
\begin{equation}\label{energy_pseudo}
E(\vec \mu)=E_{\vec\varphi}(\vec\mu)= E(\mu)+2\sum_{j=1}^p\int \varphi_j d\mu_j, \quad \mu=\sum_{j=1}^p \mu_j.
\end{equation}

Let us return to our model problem \eqref{Ort}, \eqref{C}, that is,
\begin{equation}\label{orthcondpseudo}
\int_F Q_n(z)z^k w_n(z)dz=0,\quad k=0,1,\dots,n-1,
\end{equation}
where the integration goes along a curve
\begin{equation}
\label{Cpseudo}
F \in \mc F := \{\text{rectifiable curve }   F \in \C  \text{ connecting } a_1  \text{ and }   a_2\},
\end{equation}
with $a_1, a_2 \in \C$ being two distinct   points fixed on the plane, but allowing for a more general form the weight: 
\begin{equation}
\label{w111}
w_n(z)   =   \sum_{j=1}^p e^{-2n\Phi_{n,j}(z)} g_j(z) . 
\end{equation}
For simplicity, we assume that $\Phi_{n,j}$ and $g_j$ are entire functions, and that $\Phi_{n,j}$ converge locally uniformly in $\C$ (as $n\to \infty$) to a function $\Phi_j(z)$.  

As it was observed at the end of Section~\ref{sec:numerical experiments}, conditions \eqref{orthcondpseudo}--\eqref{Cpseudo} are equivalent to
\begin{equation}\label{orthcondpseudoNew}
\sum_{j=1}^p  \int_{F_j} Q_n(z)z^k e^{-2n\Phi_{n,j}(z)} g_j(z)dz=0,\quad k=0,1,\dots,n-1,
\end{equation}
where each $F_j \in \mc F$, and curves $F_j$ can be taken independently from each other. 

This reformulation motivates the introduction of the following pseudo-vector problem: for any vector $\vec F=(F_1, \dots, F_p)\in \mc F^p$, and the vector of external fields $\vec\varphi=(\varphi_1,\dots,\varphi_p)$, $\varphi_j=\Re \Phi_j$, harmonic on $\C$, consider the energy functional \eqref{energy_pseudo} with the additional constraints that $\supp(\mu_j)\subset F_j$, $j=1, \dots, p$, and 
$$
|\mu|:=\sum_{j=1}^p |\mu_j|=\sum_{j=1}^p \int d\mu_j=1.
$$
As usual, the \textit{(pseudo-vector) equilibrium measure} $\vec \lambda=(\lambda_1,\dots,\lambda_p)=\lambda_{\vec F,\vec \varphi} $ will be the minimizer of \eqref{energy_pseudo} in this class.
It is characterized by the variational conditions
\begin{equation} \label{equilibriumPseudo}
\left(U^\lambda+\varphi_j\right)(z) \begin{cases}
=\omega, & z\in\supp(\lambda_j), \\
\ge \omega, & z\in F_j,
\end{cases}\quad j=1, \dots, p, \quad \lambda =\sum_{j=1}^p \lambda_j,
\end{equation}
(compare with \eqref{Equ-2}), where $\omega$ is the equilibrium constant.

The $S$-property in the external field $\vec\varphi$ has the form now 
\begin{equation}\label{SpropertyPseudo}
\frac{\partial (U^\lambda+\varphi_j)(z)}{\partial n_+}=\frac{\partial (U^\lambda+\varphi_j)(z)}{\partial n_-},\quad z\in \supp (\lambda_j)\setminus e, \quad j=1, \dots, p,
\end{equation}
and $\cp(e)=0$.

We have seen that equilibrium measures on sets with the $S$-property are a subclass of the critical measures, defined by \eqref{variations_critical_measure}. This subclass is proper: for each \emph{critical} measure $\vec \lambda=(\lambda_1,\dots,\lambda_p) $ we still have
$$
\left(U^\lambda+\varphi_j\right)(z) \begin{cases}
=\omega_j, & z\in\supp(\lambda_j), \\
\ge \omega_j, & z\in F_j,
\end{cases}\quad j=1, \dots, p, 
$$
but not necessarily   
\begin{equation}\label{omegaequal}
\omega_1=\dots =\omega_p.
\end{equation}
Moreover, even if \eqref{omegaequal} is true, we still need to be able to complete the supports of each $\lambda_j$ to the whole curves $F_j$ in such a way that the inequalities in \eqref{equilibriumPseudo} hold in order to claim that  $\lambda=\lambda_{\vec F,\vec \varphi}$.

Following the arguments of \cite{MR2770010,AMFSilva15} we see that the variational condition \eqref{variations_critical_measure}, when considered for functions of the form $h_z(t)=A(t)/(t-z)$, $A(z)=(z-a_1)(z-a_2)$, yield the following equation for the Cauchy transforms $C^{\lambda_j}$ of $\lambda_j$, see the definition in \eqref{principalvalue}:
\begin{equation}\label{variationalEqPseudo}
A(z)\left( \left( C^\lambda \right)^2 + 2\sum_{j=1}^p \Phi_j'  C^{\lambda_j}\right)(z)=B(z), \quad \lambda =\sum_{j=1}^p \lambda_j,
\end{equation}
where $B$ is a polynomial. This identity is valid a.e.~in $\C$ with respect to the plane Lebesgue measure.

Notice that one of the immediate conclusions is that
\begin{equation}\label{variationalEqPseudoBis}
A(z)\left(   C^\lambda  +  \Phi_j'  \right)^2(z) \in H(\C\setminus \lambda_j'), \quad j=1, \dots, p,
\end{equation}
where $\lambda_j'=\sum_{m=1,\,  m\neq j}^p \lambda_m$, and that each $\lambda_j$ lives on a trajectory of a quadratic differential, but now not necessarily rational. In order to study the structure of such trajectories it would be very convenient to associate with \eqref{variationalEqPseudo} an algebraic equation, in the spirit of what was done in \cite{AMFSilva15}, which is part of an ongoing project. 

The standard GRS theorem (Theorem~\ref{GRSthm}) is directly applicable in the pseudo-vector case when different components of the support of $\vec \lambda$ do not intersect (see, e.g., Figure~\ref{FigAssym1}, right):
\begin{thm}\label{thmnew} 
	Under the assumptions above, let polynomials $Q_n(z)=z^n+\dotsb$ be defined by the orthogonality relations \eqref{orthcondpseudo}--\eqref{w111}.
	If there exists a vector $\vec S=(S_1, \dots, S_p)$, $S_j\in \mc{F}$, with the $S$-property in the external field $\vec\varphi$, and the corresponding equilibrium measure $\lambda=\lambda_{\vec S,\vec \varphi}=(\lambda_1,\dots,\lambda_p)$  such that 
	$$
	\C\setminus \bigcup_{j=1}^p \supp(\lambda_j) \quad \text{is connected, and} \quad \supp(\lambda_j)\cap  \supp(\lambda_m)=\emptyset \quad   \text{for } j\neq m,
	$$ 
	then $\chi(Q_n)\overset{*}{ \longrightarrow}\lambda_1+\dots+\lambda_p$.
\end{thm}

When some components of the support of $\vec \lambda$ do intersect, or more precisely, when they contain a common set of positive capacity, the situation is totally open\footnote{ Strictly speaking, even the situation when the components of the support of $\vec \lambda$ are simple arcs with common endpoints is already interesting.}. It looks like in this case the max-min characterization of the support of $\vec \lambda$ dominates over its $S$-property, but at this stage these are mere speculations.
 
Meanwhile, it is difficult to obtain anything further from \eqref{variationalEqPseudo}, \emph{unless} all $\Phi_j$ are linear functions, case we consider next.

  \subsection{Piece-wise linear external field}
  
The assumption that $\varphi$ is piece-wise linear allows for a crucial simplification of the problem, using the observation that after differentiating in \eqref{variationalEqPseudoBis} all the external fields disappear. 

So, let us return to the situation \eqref{Ort}--\eqref{C}, so that $\varphi$ is \eqref{phiparticularcase}. The problem contains four parameters: $a_1, a_2 \in \C$ and $k_1, k_2 \in \R,$ or six real parameters. Finding the domain in the space of parameters for which an $S$-curve does not exist is a formidable task. It seems more convenient to start with an assumption that the $S$-curve does exist and make an ansatz on the structure of the support of its equilibrium measure; in  a piece-wise linear field  and curves in the class ${\mc F}$ the support should consist of one or two arcs. Then a straightforward method of proving existence theorems may be used, which we can shortly outline as follows: the analytic representation of the Cauchy transform of the equilibrium measure has four unknown complex parameters, constrained by a system of equations expressing the main properties of the problem.  Using any solution of such a system one  can prove existence of $S$-curve by  direct construction. In a similar situation the method has been used in
\cite{Gonchar:87}.

For the orthogonality  \eqref{Ort}--\eqref{C} and the corresponding external field \eqref{phiparticularcase} this approach yields the following result:
\begin{thm} \label{thm-3} 
	Assume that $S$  is an $S$-curve  in the external field \eqref{phiparticularcase}, that the support of the equilibrium measure $\lambda=\lambda_{S,  \varphi}$ does not intersect the imaginary axis and that it has at most two connected components. Then the following assertions are valid: 
	\begin{itemize}
		\item[(i)] The are two quadratic polynomials,
		\begin{equation}
		\label{CT-1a}
		q(t) = t^2 +c_1 t + c_0,  \quad  V(t) = (t - v_1)  (t - v_2),
		\end{equation}
		such that the Cauchy transform of the equilibrium measure $\lambda = \lambda_{\Gamma,\varphi}$ has the form
		\begin{equation} \label{CT-1}
		C^\lambda (z) = 
		\int_\infty^z  \frac {q(t)\, dt} {A(t)\sqrt{A(t) V(t)}}, \quad
		\end{equation} 
		where $ A(t) = (t - a_1)  (t - a_2)$, and the branch of the root in the denominator is fixed in $\C \setminus \supp(\lambda)$ by the condition  $\lim_{z\to \infty} \sqrt{AV(z)} /z^2 =1$.
		
    	\item[(ii)]  The parameters $v_1$, $v_2$, $ c_0$, $c_1$ are determined by the system of equations 
		\begin{align} \label{E1}
		&\int_{\infty}^{v_i} \frac {q(t)\, dt} {A(t)\sqrt{A(t) V(t)}}   = k_i, \quad i =1,2 , \\
		\label{E2}
	&	\Re \left(  \frac 1\pi  \int_{v_i}^{a_i} \frac{t-a_i}{t-a_j} \frac {q(t)\, dt} {\sqrt{A(t) V(t)}} \right) =0 ,
		\quad i, j =1,2,\ i\ne j , \\
		 \label{E3}
		&\Im \left(  \frac 1\pi  \int_{v_1}^{a_1} \frac{t-a_1}{t-a_2} \frac {q(t)\, dt} {\sqrt{A(t) V(t)}} + 
		\frac 1\pi  \int_{v_2}^{a_2} \frac{t-a_2}{t-a_1} \frac {q(t)\, dt} {\sqrt{A(t) V(t)}}\right)   = 1, \\
		\label{E4}
	&	\Re \left(  \frac 1\pi  \int_{v_1}^{v_2} \frac {q(t)\, dt} {A(t)\sqrt{A(t) V(t)}} + k_1v_1 -k_2v_2\right) =0   .
		\end{align}
			\item[(iii)]  Furthermore, $\supp(\lambda) = \Gamma_1 \cup \Gamma_2$ where $\Gamma_j$ is a critical trajectory of the quadratic differential $-R_j(z)  dz^2$, with 
		\begin{equation}
		\label{QD-1}
	R_j(z) = \left(C^\lambda + k_j\right)^2 = \left(\int_{a_j}^z  \frac {q(t)\, dt} {\sqrt{A^3(t) V(t)}}\right)^2 .
		\end{equation}
		Function $R_1$ is holomorphic in the left half plane $\C_-$, while  $R_2$   is holomorphic in $\C_+$.
	\end{itemize}
\end{thm}
\begin{proof}
It follows the lines of \cite[p.~323]{Gonchar:87}. The $S$-property, the fact that $\varphi''\equiv 0$ and the variational equation \eqref{variationalEqPseudoBis} imply the general form  \eqref{CT-1} of $C^\lambda$. Equation \eqref{E1} means that
$$
(C^\lambda + k_j)(v_j)=0, \quad j=1, 2,
$$
which follows again from \eqref{variationalEqPseudoBis}.

Equation \eqref{E2} is equivalent to
$$
\Re   \int_{v_j}^{a_j} \left(C^\lambda(t) + k_j\right) dt =0 ,
\quad j =1,2, 
$$
which is a necessary condition for the existence of trajectories joining $a_j$ with $v_j$, while \eqref{E3} is just an expression of the fact that $\lambda$ is a unit measure.
Finally, \eqref{E4} imposes the equality of the equilibrium constants, \eqref{omegaequal}.  
\end{proof}

\begin{remark}
The existence of the $S$-curve in the assumptions of Theorem~\ref{thm-3} also implies that arcs of $\supp(\lambda)$ may be completed to a curve from $\mc F$ preserving the inequalities in \eqref{equilibriumPseudo}.

On the other hands, the existence of a solution of \eqref{E1}--\eqref{E4} does not guarantee a priori the existence of the necessary critical trajectories of $-R_j(z)  dz^2$. In our case we can prove it using the homotopy arguments, starting with the totally symmetric case, analyzed next.
\end{remark}

\begin{example}[\textit{The totally symmetric case}]
Assume that
  \begin{equation}
  \label{valuesak}
  a = - a_1 = a_2 >0, \qquad k = k_1 = - k_2 > 0 
  \end{equation}
(see Figure~\ref{FigCosReal}, left, with $a=1$ and $k=0.7$), so that $\lambda_{S,\varphi}$ its just the equilibrium measure of the segment $[-a,a]$ in the external field $\varphi(x) = - k|x|$, and its existence and uniqueness is guaranteed by the classical theory, see e.g.~\cite{Saff:97}. 
  
  As usual, the main step consists in finding the support $ \Gamma_\lambda$ of $\lambda=\lambda_{S,\varphi}$, which in this case, for any values of the parameters $a$ and $k$ in \eqref{valuesak}, takes the form
   \begin{equation}
   \label{Sym-1}
   \Gamma_\lambda = [-a, -v] \cup [v, a], \quad\text{ where}\quad v = v(a, k) \in (0,a)  ;
   \end{equation}  
   the fact that $0\notin \Gamma_\lambda$ follow from \cite{MR3262924}, as well as from the calculations below.   
  Once the support is determined, we can find the expression for the Cauchy transform $C^\lambda$ of $\lambda$, and from there to recover the equilibrium measure (and by Theorem~\ref{thm2}, the asymptotics of $\chi(Q_n)$) by Sokhotsky--Plemelj's formula. Notice that the $S$-property is satisfied simply by symmetry reasons, so the existence of an $S$-curve for any such $a, k$ is now automatic.
  
The assertion of Theorem~\ref{thm-3} boils down now to
\begin{equation}
\label{Sym-2a}
C^\lambda (z) =  
\int_b^z  \frac {q(t)\, dt} {\sqrt{A^3(t) V(t)}}, \quad z \in \C\setminus \Gamma_\lambda,
\end{equation}  
where 
\begin{equation}
\label{numer}
q(z) := z^2 +c_0, \quad A(z) =  z^2 - a^2, \quad  V(z) = z^2 - v^2,
\end{equation}  
and the branch of $\sqrt{A^3 V}$ in $\C\setminus \Gamma_\lambda$ is positive for $z \in (-v,v)$.
From these expressions we get $\lambda$, which is absolutely continuous on $\Gamma_\lambda$:
 \begin{equation}
 \label{Sym-2b}
 \lambda'(x) = \lambda'(-x) =\frac 1\pi \int_b^x \frac {|q(t)| \,dt}{|\sqrt{A^3(t)  V(t)}|}
 = \frac 1{\pi i}\int_v^x \frac {q(t)\,dt}{(\sqrt{A^3(t) V(t)})_{+}}, \quad x \in [v,a],
 \end{equation} 
where as usual, $(\sqrt{A^3 V})_{+}$ are the boundary values of this function on $S_\lambda$ form the upper half plane.

The two real parameters $c_0$ and $v$ in \eqref{numer} are defined by the equations 
 \begin{equation}
 \label{Sym-3}
 \int_{-v}^v  \frac {q(t)\, dt} {\sqrt{A^3(t) V(t)}} = \int_{-v}^v (C^\lambda (x))'\, dx = 2kv, \quad \int_v^a \lambda'(x)\, dx = \frac 12 ;
 \end{equation}
 the first identity expressing the equality of the equilibrium constants on both components of $\Gamma_\lambda$, and the second one assuring that $\lambda$ is a probability measure, with its mass distributed evenly on $[-a, -v] $ and $ [v, a]$.
 
 For every $a> 0$ and  $k\geq 0$ there is a unique solution $(c_0, v)$ of this system, and hence, a unique solution of the $S$-problem for this case.  

\end{example} 

\begin{example}[\textit{A more general symmetric case}]  
Now assume that 
  \begin{equation}
  \label{Sym-1assym}
  a = - a_1 = a_2 \in \C\setminus \R, \quad \Re a>0, \qquad k = k_1 = - k_2 > 0 .
  \end{equation}

  We can no longer guarantee the existence of an $S$-curve in class $ \mc{T}_{\mc{A}}$ of curves connecting $a_1$ and $a_2$ for all values of the parameters $a, k$. However, it is intuitively clear that for $\Im a$ small enough, an $S$ curve exists, is discontinuous, and it still lies in the domain where $\varphi$ is harmonic (cf.~Figure~\ref{FigCosReal}, right, with $a=1+0.1 i$ and $k=0.7$).  
  It turns out that the formulas \eqref{Sym-1}, \eqref{Sym-2a} and \eqref{Sym-2b} remain valid with obvious modifications and equations for the parameters are similar to those in \eqref{Sym-3}.
  
  For other values of the parameters the limit distribution $\lambda$ of the zeros of $Q_n$ can be either a continuum, cutting the imaginary axis in a single (or a finite number of) point (cf.~Figure~\ref{FigAssym1}, left, with $a_2=-a_1=0.5+0.75 i$ and $k=0.13$), or it can split into disjoint arcs, each or some of them crossing the imaginary axis (cf.~Figure~\ref{FigAssym1}, right, with $a_2=-a_1=0.5+0.75 i$ and $k=1$). In this case we expect these arcs to be the support of the equilibrium measure for the pseudo-vector problem described in Section~\ref{sec:pseudo}.
  
\end{example}

We conclude our exposition with the remark that the reformulation of the orthogonality conditions \eqref{orthcondpseudo}--\eqref{Cpseudo} as \eqref{orthcondpseudoNew} could be crucial for the application of the nonlinear steepest descent analysis of Deift-Zhou \cite{MR2000g:47048} for the study of the strong (and in consequence, also weak-*) asymptotics of $Q_n$'s, especially in the most elusive cases when the $S$-property is apparently lost. We plan to address this problem in a forthcoming publication.


\section*{Acknowledgments}
This paper is based on a plenary talk at the FoCM conference in Montevideo, Uruguay, in 2014 by the first author (AMF), who is very grateful to the organizing committee for the invitation. This work was completed during a stay of AMF as a Visiting Chair Professor at the Department of Mathematics of the Shanghai Jiao Tong University (SJTU), China. He acknowledges the hospitality of the host department and of the SJTU.

AMF was partially supported by the Spanish Government together with the European Regional Development Fund (ERDF) under grants
MTM2011-28952-C02-01 (from MICINN) and MTM2014-53963-P (from MINECO), by Junta de Andaluc\'{\i}a (the Excellence Grant P11-FQM-7276 and the research group FQM-229), and by Campus 
de Excelencia Internacional del Mar (CEIMAR) of the University of Almer\'{\i}a.  

We are grateful to S.~P.~Suetin for allowing us to use the results of his calculations in Figure~\ref{Fig1HermitePade}, as well as to an anonymous referee for several remarks that helped to improve the presentation.



%

\def\cprime{$'$}

\end{document}